\documentclass[amscd,amssymb,11pt]{amsart}

\usepackage{amssymb}
\usepackage{graphicx}
\usepackage[table]{xcolor} \usepackage[all]{xy}

\usepackage{hyperref}
\usepackage{standalone} \usepackage{ytableau} \usepackage{diagbox}
\newtheorem{thm}{Theorem}[section]
\newtheorem{lem}[thm]{Lemma}
\newtheorem{cor}[thm]{Corollary}
\newtheorem{prop}[thm]{Proposition}
\newtheorem{conj}[thm]{Conjecture}

\theoremstyle{definition}

\newtheorem{ex}[thm]{Example}

\theoremstyle{remark}

\newtheorem{rem}[thm]{Remark}

\numberwithin{equation}{section}

\newcommand{\thmref}[1]{Theorem~\ref{#1}}
\newcommand{\corref}[1]{Corollary~\ref{#1}}
\newcommand{\secref}[1]{\S\ref{#1}}
\newcommand{\conjref}[1]{Conjecture~\ref{#1}}
\newcommand{\propref}[1]{Proposition~\ref{#1}}
\newcommand{\lemref}[1]{Lemma~\ref{#1}}

\newcommand{\Hom}{\operatorname{Hom}}

\newcommand{\End}{\operatorname{End}}

\newcommand{\im}{\operatorname{im}}

\newcommand{\Gr}{\operatorname{Gr}}

\newcommand{\A}{{\mathcal  A}}

\newcommand{\Z}{{\mathbb  Z}}

\newcommand{\Q}{{\mathbb  Q}}

\newcommand{\R}{{\mathbb  R}}
\newcommand{\C}{{\mathbb  C}}

\newcommand{\D}{{\mathbb  D}}

\newcommand{\sm}{\wedge}

\newcommand{\ra}{\rightarrow}
\newcommand{\xra}{\xrightarrow}

\newcommand{\hra}{\hookrightarrow}

\DeclareMathOperator{\content}{c}
\DeclareMathOperator{\connectedcomponents}{cc}
\DeclareMathOperator{\corners}{C}

\begin{document}

\title[K-theory of Grassmanians]{Computing the Morava K-theory of real Grassmanians using chromatic fixed point theory.}

\author[Kuhn]{Nicholas J.~Kuhn}
\email{njk4x@virginia.edu}

\author[Lloyd]{Christopher J.~R.~Lloyd}
\email{cjl8zf@virginia.edu}

\address{Department of Mathematics \\ University of Virginia \\ Charlottesville, VA 22903}

\thanks{The first author is a PI of RTG NSF grant DMS-1839968, which partially supported the research of the second author.}

\date{November 16, 2021.}


\begin{abstract}
We study $K(n)^*(\Gr_d(\R^m))$, the 2-local Morava $K$-theories of the real Grassmanians, about which very little has been previously computed.  We conjecture that the Atiyah-Hirzebruch spectral sequences computing these all collapse after the first possible non-zero differential $d_{2^{n+1}-1}$, and give much evidence that this is the case.

We use a novel method to show that higher differentials can't occur:  we get a lower bound on the size of $K(n)^*(\Gr_d(\R^m))$ by constructing a $C_4$--action on our Grassmanians and then applying the chromatic fixed point theory of the authors' previous paper.  In essence, we bound the size of $K(n)^*(\Gr_d(\R^m))$ by computing $K(n-1)^*(\Gr_d(\R^m)^{C_4})$.

Meanwhile the size of $E_{2^{n+1}}$ is given by $Q_n$--homology, where $Q_n$ is Milnor's $n$th primitive mod 2 cohomology operation.  Whenever we are able to calculate this $Q_n$--homology, we have found that the size of $E_{2^{n+1}}$ agrees with our lower bound for the size of $K(n)^*(\Gr_d(\R^m))$.  We have two general families where we prove this: $m\leq 2^{n+1}$ and all $d$, and  $d=2$ and all $m$ and $n$.  Computer calculations have allowed us to check many other examples with larger values of $d$.

\end{abstract}

\maketitle

\newpage

\section{Introduction} \label{introduction}

Let $\Gr_d(\R^m)$ be the real Grassmanian of $k$-planes in $\R^m$, a much studied compact manifold of dimension $d(m-d)$ admitting the structure of a CW complex with $\binom{m}{d}$ `Schubert cells'.

Much is known about the ordinary cohomology of these spaces:
\begin{enumerate}
\item $H^*(\Gr_d(\R^{m});\Z/2)$ is generated by Stiefel--Whitney classes satisfying standard relations.  It has total dimension $\binom{m}{d}$.

\item $H^*(\Gr_d(\R^{m});\Q)$ is generated by Pontryagin classes, along with, in some cases, an odd dimensional class. For fixed $d$, and $\epsilon = $ 0 or 1, the total dimension of $H^*(\Gr_d(\R^{2-\epsilon + 2l});\Q)$ is polynomial of degree $\left\lfloor d/2 \right\rfloor$ as a function of $l \geq 0$.

\item If $m$ is even, then $\Gr_d(\R^m)$ is oriented. Furthermore, the inclusion $\Gr_d(\R^{m-1})\hra \Gr_d(\R^{m})$ induces an epimorphism in rational cohomology.

\item Nontrivial torsion in $H^*(\Gr_d(\R^{m});\Z)$ has order 2. The mod 2 Bockstein Spectral Sequence (BSS) collapses after the first differential. Equivalently, the mod 2 Adams Spectral Sequence (ASS) converging to $H^*(\Gr_d(\R^{m});\Z)$ collapses at $E_2$.
\end{enumerate}

Much less is known about other cohomology theories applied to these Grassmanians.  In this paper, we study $K(n)^*(\Gr_d(\R^{m}))$ for $n \geq 1$. Here $K(n)^*(X)$ denotes the 2-local $n$th Morava $K$-theory of a space $X$, a graded vector space over the graded field $K(n)_*= \Z/2[v_n^{\pm}]$ with $|v_n|=2^{n+1}-2$.  We let $k(n)$ denote the connective cover of $K(n)$: $k(n)_* = \Z/2[v_n]$.

Viewing $H\Q$ as $K(0)$ and $H\Z$ as $k(0)$, our discovery is that analogues of statements (2)--(4) above appear to hold for all $n$, with the Atiyah-Hirzebruch Spectral Sequence (AHSS) replacing the Bockstein Spectral Sequence in statement (4). Furthermore, the analogue of statement (1) holds through a much bigger range than one would expect from dimension considerations.

In the next two subsections, we describe our main results.

\subsection{Results proved using chromatic fixed point theory}
\mbox{}
\vspace{.1in}

Given a finite complex $X$ and $n\geq 0$, we let $k_n(X) = \dim_{K(n)_*}K(n)^*(X)$.

\begin{thm} \label{collapse thm} If $m \leq 2^{n+1}$, then $k_n(\Gr_d(\R^m)) = \binom{m}{d}$.  Thus, in this range, the AHSS converging to $K(n)^*(\Gr_d(\R^m))$ collapses at $E_2$.
\end{thm}

We note that this collapsing range is surprisingly large, as dimension considerations just imply collapsing if $d(m-d) < 2^{n+1}$.

For larger $m$, we have the following lower bound.

\begin{thm} \label{lower bound thm} Let $m = 2^{n+1} - \epsilon + 2l$ with $\epsilon =$ 0 or 1, and $l \geq 0$.  Then
$$ k_n(\Gr_d(\R^m)) \geq \sum_{i=0}^{\left\lfloor d/2 \right\rfloor} \binom{2^{n+1}-\epsilon}{d-2i}\binom{l}{i}.$$
\end{thm}

\begin{conj} \label{kn conj} Equality always holds in this last theorem.
\end{conj}

The biggest novelty of this paper is our method for proving Theorems \ref{collapse thm} and \ref{lower bound thm}: we make use of chromatic fixed point theory to prove these nonequivariant results.

The blue shift theorem of \cite{6 author} says that if $C$ is a finite cyclic $p$--group and $X$ is a finite $C$-CW complex, then
$$ \widetilde K(n)^*(X) = 0 \Rightarrow \widetilde K(n-1)^*(X^C) = 0.$$
In \cite{kuhn lloyd 2020}, we upgraded this as follows.

\begin{thm} \cite[Theorem 2.16]{kuhn lloyd 2020} \label{chromatic fixed point thm} If $C$ is a finite cyclic $p$--group, and $X$ is a finite $C$-CW complex, then
$$ k_n(X) \geq k_{n-1}(X^C).$$
\end{thm}
(In these statements, $K(n)_*$ means Morava $K$-theory at the prime $p$.) \\

As $\binom{m}{d}$ is an evident upper bound for $k_n(\Gr_d(\R^m))$, to prove \thmref{collapse thm}, it suffices to show that $k_n(\Gr_d(\R^m)) \geq \binom{m}{d}$ in the stated range. Using \thmref{chromatic fixed point thm}, we will show this by induction on $n$ using a $C_2$--action on $\Gr_d(\R^m)$ induced by an $m$--dimensional real representation of $C=C_2$.

We will similarly prove \thmref{lower bound thm} for $n \geq 1$ by using a $C_4$--action on $\Gr_d(\R^m)$ induced by an $m$--dimensional real representation of $C=C_4$.

In both cases, it will be quite easy to compute $k_{n-1}(\Gr_d(\R^m)^C)$.

Details of this will be in \secref{chromatic section}.

\subsection{Results about the $Q_n$--homology of the Grassmanians}
\mbox{}
\vspace{.1in}

Conjecture \ref{kn conj} follows from a conjectural calculation that only involves $H^*(\Gr_d(\R^m); \Z/2)$, viewed as a module over the Steenrod algebra.

Let $Q_n$, $n = 0, 1, 2, \dots$, be the Milnor primitives: the elements in the mod 2 Steenrod algebra recursively defined by $Q_0 = Sq^1$, and $Q_n = [Q_{n-1}, Sq^{2^n}]$.  These satisfy $Q_n^2=0$, and we let $k_{Q_n}(X)$ denote the total dimension of the {\em $Q_n$--homology of $X$}:
$$ H^*(X;Q_n) = \frac{Z^*(X;Q_n)}{B^*(X;Q_n)},$$
where $Z^*(X;Q_n) = \ker \{Q_n: H^*(X;\Z/2) \ra H^{*+2^{n+1}-1}(X;\Z/2)\}$ and
$B^*(X;Q_n) = \im \{Q_n: H^{*-2^{n+1}+1}(X;\Z/2) \ra H^*(X;\Z/2)\}$.

As will be reviewed in \secref{AHSS = ASS subsection}, the first differential in the AHSS converging to $K(n)^*(X)$ is $ d_{2^{n+1}-1}$, with formula
$$ d_{2^{n+1}-1}(x) = Q_n(x)v_n,$$
for all $x \in E_2^{*,0}(X)= H^*(X;\Z/2)$.  This makes it not hard to check the next lemma.

\begin{lem} \label{AHSS & ASS lem} If $X$ is a finite complex, $k_{Q_n}(X) \geq k_n(X)$ is always true, and the following are equivalent:

 (a) $k_{Q_n}(X) = k_n(X)$.

 (b) The AHSS, when $n\geq 1$, or the BSS, when $n=0$, computing $K(n)^*(X)$ collapses at $E_{2^{n+1}}$.

 (c) The ASS computing $k(n)^*(X)$ collapses at $E_2$.
\end{lem}

We apply this to our situation.  Firstly, \thmref{collapse thm} has the following nontrivial algebraic consequence.

\begin{cor}  If $m \leq 2^{n+1}$, then $Q_n$ acts trivially on $H^*(\Gr_d(\R^m);\Z/2)$.
\end{cor}

(For an algebraic proof of this result using the methods of \secref{lenart formula sec}, see the second author's thesis \cite[p.75]{lloyd thesis}.)

For $m>2^{n+1}$, we believe the following is true.

\begin{conj} \label{collapsing conj} Let $m = 2^{n+1} - \epsilon + 2l$ with $\epsilon =$ 0 or 1, and $l \geq 0$.  Then
$$ k_{Q_n}(Gr_d(\R^m)) = \sum_{i=0}^{\left\lfloor d/2 \right\rfloor} \binom{2^{n+1}-\epsilon}{d-2i}\binom{l}{i}.$$
\end{conj}

Comparison with \thmref{lower bound thm} shows that when \conjref{collapsing conj} is true, one can conclude
\begin{itemize}
\item $k_{Q_n}(\Gr_d(\R^m)) = k_n(\Gr_d(\R^m))$, and \conjref{kn conj} is true.
\item The AHSS computing $K(n)^*(\Gr_d(\R^m))$ collapses at $E_{2^{n+1}}$.
\item The ASS computing $k(n)^*(\Gr_d(\R^m))$ collapses at $E_2$.
\item $k_n(\Gr_d(\R^{2^{n+1}-\epsilon + 2l}))$ is polynomial of degree $\left\lfloor d/2 \right\rfloor$ as a function of $l$.
\end{itemize}

Known rational calculations imply that the conjecture is true when $n=0$.  It is also easy to show that the conjecture is true when $d=1$, and one calculates
\begin{equation*}
k_n(\Gr_1(\R^m)) =
\begin{cases}
m & \text{if } 1 \leq m \leq 2^{n+1} \\  2^{n+1} - \epsilon &  \text{if } m = 2^{n+1}-\epsilon + 2l.
\end{cases}
\end{equation*}

With much more work we prove the following.

\begin{thm} \label{d=2 thm}  \conjref{collapsing conj} is true when $d = 2$.  Thus the Atiyah--Hirzebruch spectral sequence computing $K(n)_*(\Gr_2(\R^m))$ collapses at $E_{2^{n+1}}$, the Adams spectral sequence computing $k(n)_*(\Gr_2(\R^m))$ collapses at $E_2$, and we have the calculation
\begin{equation*}
k_n(\Gr_2(\R^m)) =
\begin{cases}
\binom{m}{2} & \text{if } 2 \leq m \leq 2^{n+1} \\  \binom{2^{n+1}-\epsilon}{2} + l & \text{if }  m = 2^{n+1}-\epsilon + 2l.
\end{cases}
\end{equation*}
\end{thm}
\bigskip

We are firm believers in our conjectures.  For more evidence, the second author has made extensive computer calculations verifying Conjecture \ref{collapsing conj} in hundreds more cases with larger values of $d$: see the tables in the Appendix. \\

For $d\geq 2$, computing the size of $H^*(\Gr_d(\R^m);Q_n)$ seems tricky. We have organized our efforts by studying how these numbers change as $m$ is increased as follows.

Let $C_d(\R^{m})$ denote the cofiber of the inclusion $\Gr_d(\R^{m-1}) \ra \Gr_d(\R^{m})$, so there is a cofiber sequence $\Gr_d(\R^{m-1}) \xra{i} \Gr_d(\R^{m}) \xra{p} C_d(\R^m)$.  In \secref{cofib subsection}, $C_d(\R^{m})$ is identified as the Thom space of the canonical normal bundle over $\Gr_{d-1}(\R^{m-1})$, and in \secref{twisted subsection}, we study the $Q_n$--module $\widetilde H^*(C_d(\R^{m});\Z/2)$, viewed as $H^*(\Gr_{d-1}(\R^{m-1}); \Z/2)$ equipped with an explicit twisted $Q_n$--action.  One has an induced short exact sequence of modules over the Steenrod algebra
$$0 \ra \widetilde H^*(C_d(\R^{m}); \Z/2) \xra{p^*} H^*(\Gr_d(\R^m); \Z/2) \xra{i^*} H^*(\Gr_d(\R^{m-1}); \Z/2) \ra 0,$$
inducing a long exact sequence on $Q_n$--homology. \\

When $m$ is even, we see much orderly behavior.

\begin{thm} \label{even m properties thm} Let $m$ be even.

(a) $H^{d(m-d)}(\Gr_d(\R^m);Q_n) \simeq \Z/2$, i.e. the nonzero top dimensional cohomology class is not in the image of $Q_n$ for all $n$.

(b) The chain complex $(\widetilde H^*(C_d(\R^{m});Q_n))$ is dual to the chain complex $(H^{d(m-d)-*}(Gr_{d-1}(\R^{m-1});Q_n))$.

(c) If \conjref{collapsing conj} is true for $(n,d,m-1)$ and $(n,d-1,m-1)$, then it is true for $(n,d,m)$.  Furthermore, $\Gr_d(\R^m)$ will then be $k(n)$--oriented, and the cofiber sequence above will induce short exact sequences
$$ 0 \ra \widetilde H^*(C_d(\R^m); Q_n) \xra{p^*} H^*(\Gr_d(\R^m); Q_n) \xra{i^*} H^*(Gr_d(\R^{m-1}); Q_n) \ra 0, $$
$$ 0 \ra \widetilde K(n)^*(C_d(\R^m)) \xra{p^*} K(n)^*(\Gr_d(\R^m)) \xra{i^*} K(n)^*(\Gr_d(\R^{m-1})) \ra 0. $$
\end{thm}

We prove \thmref{even m properties thm} in \secref{m even section}. We make use of the additive basis $\{s_{\lambda}\}$ dual to the classical Schubert cells. Here $\lambda$ runs through partitions having at most $d$ parts, each no bigger than $m-d$.  In \cite{lenart}, Cristian Lennart gave a combinatorial formula for $Q_n(s_\lambda)$, and we use this to prove (a).  Duality statement (b) follows quite formally from (a), and (c) follows easily from (b). \\

When $m$ is odd, the analogues of statements (a) and (b) are false, and, for $d \geq 3$, the full behavior of the connecting map in the $Q_n$-homology long exact sequence,
$$ \delta: H^*(\Gr_d(\R^{m-1});Q_n) \ra \widetilde H^{*+2^{n+1}-1}(C_d(\R^{m});Q_n),$$
is as yet unclear to the authors.  In \secref{speculation section}, we will prove analogues of \thmref{collapse thm} and \thmref{lower bound thm} for $C_d(\R^m)$, and then speculate on behavior of $\delta$ that would be compatible with all of our computations.

However, when $d=2$, we have the following result.

\begin{thm} \label{d=2 odd m boundary thm}  Let $m>2^{n+1}$ be odd.  Then $k_{Q_n}(C_2(\R^{m})) = 2^{n+1}-2$ and the map
$$\widetilde H^{*}(C_2(\R^{m});Q_n) \xra{p^*} H^{*}(\Gr_2(\R^{m});Q_n)$$
is zero, so that there is a short exact sequence
$$ 0 \ra H^*(\Gr_2(\R^m); Q_n) \xra{i^*} H^*(\Gr_2(\R^{m-1}); Q_n) \xra{\delta} \widetilde H^{*+2^{n+1}-1}(C_2(\R^{m});Q_n) \ra 0, $$
\end{thm}
From this \thmref{d=2 thm} quickly follows and one can deduce that, in this case, there is a short exact sequence
$$ 0 \ra  K(n)^*(\Gr_2(\R^m)) \xra{i^*} K(n)^*(\Gr_2(\R^{m-1})) \xra{\delta} \widetilde K(n)^{*+1}(C_2(\R^m)) \ra 0. $$

We prove \thmref{d=2 odd m boundary thm} in \secref{d=2 section}. The tools we use are very different from those used in proving \thmref{even m properties thm}: we work with the classical presentation of $H^*(\Gr_d(\R^m);\Z/2))$ as a ring of Stiefel--Whitney classes.

\subsection{Comparison with other work}
\mbox{}
\vspace{.1in}

When comparing our work to what has come before, the first thing to say is that the outcome of our calculations -- though not the methods -- are in line with the classical calculations first made by C.Ehresmann in 1937 \cite{ehresmann}.  He determined the additive structure of both $H^*(\Gr_d(\R^m);\Z/2)$ and $H^*(\Gr_d(\R^m);\Q)$.  He also showed that all the torsion in $H^*(\Gr_d(\R^m);\Z)$ was of order 2; in modern terms this is equivalent to showing that the Bockstein spectral sequence computing $H^*(\Gr_d(\R^m);\Z)$ collapses after the first nonzero differential given by $Q_0 = Sq^1 = \beta$.

Calculating the Morava $K$--theories of $\Gr_d(\R^{\infty}) = BO(d)$ was done first by Kono and Yagita \cite{kono yagita}, and then, with a simpler proof, by Kitchloo and Wilson \cite{kitchloo wilson}.  Again, the AHSS computing $K(n)^*(BO(d))$ collapses after the first nonzero differential, but the collapsing is for an elementary reason: $H^*(BO(d);Q_n)$ is concentrated in even degrees.  Indeed, one quickly learns that the complexification map $BO(d) \ra BU(d)$ induces an epimorphism $K(n)^*(BU(d)) \ra K(n)^*(BO(d))$, so that $K(n)^*(BO(d))$ is generated by Chern classes $c_1, \dots, c_d$.

An equivalent statement is that  $H^*(BO(d);Q_n)$ is generated by the classes $w_1^2, \dots, w_d^2$.  These will still be permanent classes in the AHSS converging to $K(n)^*(\Gr_d(\R^m))$, but now we have odd dimensional classes as well, with the number of these seemingly growing as $d$ and $m$ grow.

Finally, we point out that we do not attempt to describe $K(n)^*(\Gr_d(\R^m))$ as a $K(n)^*$--algebra.  Our results do tell us something about this, however.  In the situation of \thmref{collapse thm}, the known algebra $H^*(\Gr_d(\R^m); \Z/2) \otimes K(n)^*$ will be an associated graded. Similarly, whenever our conjecture is valid, $H^*(\Gr_d(\R^m); Q_n) \otimes K(n)^*$ would be an associated graded of the $K(n)^*$--algebra $K(n)^*(\Gr_d(\R^m))$.  What is still needed, and might be necessary to prove our conjectural collapsing in general, are sensible constructions of classes in odd degrees.

\section{The proofs of Theorems \ref{collapse thm} and \ref{lower bound thm}} \label{chromatic section}

In this section we prove Theorems \ref{collapse thm} and \ref{lower bound thm} by using our chromatic fixed point theorem \thmref{chromatic fixed point thm}.

\subsection{A fixed point formula}
\mbox{}
\vspace{.1in}

Let $G$ be a finite group, and let $V$ be an $m$--dimensional real representation of $G$.  Then $Gr_d(V)$, the space of $d$--planes in $V$, is a model for $\Gr_d(\R^m)$ with an evident $G$-action.   Here we describe $\Gr_d(V)^G$, its space of $G$-fixed points.

To state this, we need some notation.  Let $V_1, \dots, V_k$ be the irreducible real representations of $G$, let $r_i = \dim_{\R} V_i$, and let $\D_i = \End_{\R[G]}(V_i,V_i)$.  Each of the endomorphism algebras $\D_i$ will be a finite dimensional real division algebra, and thus isomorphic to $\R$, $\C$, or $\mathbb H$, and $\dim_{\R} \D_i$ will divide $r_i$.

\begin{prop}\label{fixed point formula}  If $\displaystyle V = V_1^{m_1} \oplus \cdots \oplus V_k^{m_k}$, then there is a homeomorphism
$$ \Gr_d(V)^G = \bigsqcup_{j_1 r_1 + \ldots + j_k r_k=d}
    \Gr_{j_1}(\D_1^{m_1})  \times \cdots
    \times \Gr_{j_k}(\D_k^{m_k}).$$
\end{prop}

\begin{proof} The fixed point space $\Gr_d(V)^G$ will be the space of sub-$G$-modules $W < V$ of real dimension $d$.  Such a $G$--module $W$ will decompose canonically as $W = W_1 \oplus \dots \oplus W_k$, with $W_i < V_i^{m_i}$.  If $d_i = \dim_{\R}W_i$, then $d_1 + \dots + d_k=d$. Thus we have a decomposition
$$ \Gr_d(V)^G = \bigsqcup_{d_1 + \ldots + d_k=d}
    \Gr_{d_1}(V_1^{m_1})^G \times \Gr_{d_2}(V_2^{m_2})^G \times \cdots
    \times \Gr_{d_k}(V_k^{m_k}).$$

A submodule $W_i$ of $V_i^{m_i}$ must be isomorphic to $V_i^{j}$ for some $j$, thus $\Gr_{d_i}(V_i^{m_1})^G$ will be empty unless $d_i = j_ir_i$ for some $j_i$.

Finally, using that $\Hom_{\R[G]}(V_i^{j_i}, V_i^{m_i}) = \Hom_{\D}(\D^{j_i},\D^{m_i})$, one deduces that the submodules of $V_i^{m_i}$ isomorphic to $V_i^{j_i}$ correspond to the $\D$-subspaces of $\D^{m_i}$ of dimension $j_i$ over $\D$.  Thus there is a homeomorphism
$$\Gr_{j_ir_i}(V_i^{m_i})^G = \Gr_{j_i}(\D_i^{m_i}).$$
\end{proof}

\begin{cor} \label{kn fixed point formula} If $\displaystyle V = V_1^{m_1} \oplus \cdots \oplus V_k^{m_k}$, then, for any $n$,
$$ k_n(\Gr_d(V)^G) = \sum_{j_1 r_1 + \ldots + j_k r_k=d}
    k_n(\Gr_{j_1}(\D_1^{m_1})) \cdots
    k_n(\Gr_{j_k}(\D_k^{m_k})).$$
\end{cor}
\begin{proof}  A consequence of the Kunneth theorem for $K(n)_*$ is that $k_n(X \times Y) = k_n(X)k_n(Y)$.  Thus the corollary follows from the proposition.
\end{proof}

\begin{rem}  If $\D = \C \text{ or } \mathbb H$, then $\Gr_d(\D^m)$ has a CW structure with $\binom{m}{d}$ cells that are all even dimensional, and thus $k_n(\Gr_d(\D^m)) = \binom{m}{d}$ for all $n$.
\end{rem}

\subsection{Proof of \thmref{collapse thm}}
\mbox{}
\vspace{.1in}

\thmref{collapse thm} says that if $m \leq 2^{n+1}$ then $k_n(\Gr_d(\R^m)) = \binom{m}{d}$. Using \thmref{chromatic fixed point thm} and \propref{fixed point formula}, we prove this by induction on $n$.

The $n=0$ case of the theorem is easy to check as
$$
\Gr_d(\R^0) =
\begin{cases}
* & \text{if } d = 0 \\ \emptyset & \text{otherwise}
\end{cases}
\text{ and }
\Gr_d(\R^1) =
\begin{cases}
* & \text{if } d = 0,1 \\ \emptyset & \text{otherwise}
\end{cases}.
$$

For the inductive step, assume that if $p \leq 2^{n}$ then $k_{n-1}(\Gr_d(\R^p)) = \binom{p}{d}$.

Let $m \leq 2^{n+1}$.  As it is clear that $k_n(\Gr_d(\R^m)) \leq \binom{m}{d}$, our goal is to show that $k_n(\Gr_d(\R^m)) \geq \binom{m}{d}$.

Let $C_2$ be the cyclic group of order 2.  To get our needed lower bound, our strategy will be to make $\R^m$ into a $C_2$--module, and then apply \thmref{chromatic fixed point thm}.

The group $C_2$ has two irreducible 1--dimensional real representations: call them $L_1$ and $L_2$.  Since $m \leq 2^{n+1}$, we can write $m$ as $m=p+q$ with both $p \leq 2^n$ and $q \leq 2^n$.  Now let $V = L_1^p \oplus L_2^q$, an $m$--dimensional real representation of $C_2$.

Applying \propref{fixed point formula}, we see that
$$ \Gr_d(V)^{C_2} = \bigsqcup_{i+j=d} \Gr_i(\R^p) \times \Gr_j(\R^q).$$

Applying \thmref{chromatic fixed point thm} to this, we learn that
\begin{equation*}
\begin{split}
k_n(\Gr_d(\R^m)) &
\geq  \sum_{i+j=d} k_{n-1}(\Gr_i(\R^p))k_{n-1}(\Gr_j(\R^q)) \\
  & = \sum_{i+j=d} \binom{p}{i}\binom{q}{j} \text{ \ (by inductive hypothesis)} \\
  & = \binom{m}{d}.
\end{split}
\end{equation*}

\begin{rem}  The same inductive proof can be used to prove the classical result that $\dim_{\mathbb Z/2} H^*(\Gr_d(\R^m); \mathbb Z/2) = \binom{m}{d}$ for all $m$ and $d$, with our chromatic fixed point theorem \thmref{chromatic fixed point thm} replaced by the classical theorem of Ed Floyd \cite[Theorem 4.4]{floyd TAMS 52}: if the cyclic group $C_p$ acts on a finite CW complex $X$, then
$ \dim_{\mathbb Z/p} H^*(X; \mathbb Z/p) \geq \dim_{\mathbb Z/p} H^*(X^{C_p}; \mathbb Z/p)$.  It would be interesting to know if this argument was known to Floyd, or others, like Bob Stong, who regularly worked with these sorts of group actions.

\end{rem}

\subsection{Proof of \thmref{lower bound thm}}
\mbox{}
\vspace{.1in}

The strategy of the proof of \thmref{lower bound thm} is the same as the proof in the last subsection: we get a lower bound on $k_n(\Gr_d(\R^m))$ by letting a cyclic 2--group act on $\R^m$ and applying \thmref{chromatic fixed point thm}.

In this case, the representation theory of $C_2$ is not rich enough to give us a big enough lower bound, but a well chosen real representation of the group $C_4$ of order 4 works better.  Curiously, in our calculation of $k_{n-1}$ of the resulting fixed point space, we are able to use our already proven \thmref{collapse thm}, so the proof is not by induction, but more direct.

The group $C_4$ has three irreducible real representations:  $L_1$ and $L_2$ of dimension 1, and $R$ of real dimension 2. Note that $\End_{\R[C_4]}(R) \simeq \C$.

Now let $m = 2^{n+1} - \epsilon + 2l$ with $\epsilon =$ 0 or 1, and $l \geq 0$.  We define an $m$ dimensional real representation $V$ of $C_4$ by $V = L_1^{2^n} \oplus L_2^{2^n-\epsilon} \oplus R^l$.

Applying \propref{fixed point formula}, we see that
$$ \Gr_d(V)^{C_4} = \bigsqcup_{j+k+2i=d} \Gr_j(\R^{2^n}) \times \Gr_k(\R^{2^n-\epsilon})\times \Gr_i(\C^l).$$

Applying \thmref{chromatic fixed point thm} to this, we learn that
\begin{equation*}
\begin{split}
k_n(\Gr_d(\R^m)) &
\geq  \sum_{j+k+2i=d} k_{n-1}(\Gr_j(\R^{2^n}))k_{n-1}(\Gr_k(\R^{2^n-\epsilon}))k_{n-1}(\Gr_i(\C^{l})) \\
  & = \sum_{j+k+2i=d} \binom{2^n}{j}\binom{2^n-\epsilon}{k}\binom{l}{i} \text{ \ (using \thmref{collapse thm})} \\
  & = \sum_i \left [ \sum_{j+k=d-2i} \binom{2^n}{j}\binom{2^n-\epsilon}{k} \right ] \binom{l}{i} \\
  & = \sum_i \binom{2^{n+1}-\epsilon}{d-2i}\binom{l}{i}.
\end{split}
\end{equation*}

\section{The $Q_n$ homology of $Gr_d(\R^m)$: background material} \label{Qn section}

\subsection{The AHSS and the ASS for Morava $K$--theory} \label{AHSS = ASS subsection}
\mbox{}
\vspace{.1in}

Let $n \geq 1$.  We recall the structure of the AHSS converging to $K(n)^*(X)$ (as always, in this paper, with $p=2$). It is a spectral sequence of graded $K(n)^* = \Z/2[v_n^{\pm}]$ algebras with
$$ E_2^{*,\star}(X) = H^*(X;K(n)^\star) = H^*(X;\Z/2)[v_n^{\pm}].$$
Here $v_n$ has cohomological degree $2-2^{n+1}$.

Sparseness of the rows implies that the differential $d_r$ will be zero unless $r = s(2^{n+1}-2)+1$ for some $s$.  The first possible nonzero differential, $d_{2^{n+1}-1}$, satisfies the following formula \cite{Yagita 80}: given $x \in E_2^{*,0}(X)= H^*(X;\Z/2)$,
$$ d_{2^{n+1}-1}(x) = Q_n(x)v_n.$$

It follows that $E_{2^{n+1}}(X) \simeq H^*(X;Q_n)[v_n^{\pm}]$, and so the dimension of $E_{2^{n+1}}(X)$ as a $K(n)^*$--vector space will equal $k_{Q_n}(X)$, the dimension of the $Q_n$--homology of $X$.  One immediately deduces part of \lemref{AHSS & ASS lem}: $k_{Q_n}(X) = k_n(X)$ if and only if the AHSS converging to $K(n)^*(X)$ collapses at $E_{2^{n+1}}(X)$.

To continue with the proof of \lemref{AHSS & ASS lem}, let $cE_r^{*,\star}(X)$ denote the terms of the AHSS computing $k(n)^*(X)$, a 4th quadrant spectral sequence.  Note that $cE_2^{*,\star} = H^*(X;\Z/2)[v_n]$ embeds in $E_2^{*,\star}(X) = H^*(X;\Z/2)[v_n^{\pm}]$, and equals it for $\star\leq 0$, and that the latter spectral sequence is obtained from the former by inverting $v_n$.

It follows that $cE_{2^{n+1}}^{*,\star}(X)= E_{2^{n+1}}^{*,\star}(X)$ for $\star < 0$, with the map on the 0-line between the spectral sequences corresponding to the epimorphism $Z^*(X;Q_n) \twoheadrightarrow H^*(X;Q_n)$.  From this, one sees that any higher differential in the $k(n)^*(X)$ AHSS would be detected in the $K(n)^*(X)$ AHSS. Since this 2nd spectral sequence is the localization of the first, we an conclude that the $K(n)^*(X)$ AHSS collapses at $E_{2^{n+1}}(X)$ if and only if the $k(n)^*(X)$ AHSS collapses at $cE_{2^{n+1}}(X)$.

Next we note that the AHSS spectral sequence $cE_r^{*,\star}(X)$ identifies with the ASS computing $k(n)^*(X)$ with suitable re-indexing, with $cE_{2^{n+1}}^{*,\star}(X)$ corresponding to the Adams $E_2$ term.  Firstly, a result of C.R.F.Maunder \cite{maunder 1963} implies that the AHSS converging to $[X,k(n)]_*$ can be constructed by taking the Postnikov filtration of the spectrum $k(n)$. But the Postnikov tower for $k(n)$ is also an Adams tower: as described in in the survey paper \cite[\S 5]{Wurgler 91}, there is a cofibration sequence
$$ \Sigma^{2^{n+1}-2} k(n) \xra{v_n} k(n) \xra{\pi} H\Z/2 \xra{\bar Q_n} \Sigma^{2^{n+1}-1} k(n)$$ such that $\Sigma^{2^{n+1}-1}\pi \circ \bar Q_n = Q_n$ and $\pi$ induces the epimorphism $A \ra A/AQ_n$ on mod 2 cohomology.

Finally, we note that, when $n=0$, one still has the cofibration sequence as above, with now with $v_0 = 2$, so that the ASS for $k(0) = H\Z$ is similarly related to the Bockstein spectral sequence.

\subsection{The description of $H^*(\Gr_d(\R^m);\Z/2)$ via Stiefel--Whitney classes.}
\mbox{}

We recall classical results that are either explicitly in \cite{milnor stasheff} or can easily be deduced from the material there.

Let \(w_1,\ldots,w_d\) denote the Stiefel--Whitney classes of the
canonical $d$--dimensional bundle $\gamma_d$ over \(\Gr_d(\R^{\infty})\).  One has
$$ H^*(\Gr_d(\R^{\infty});\Z/2) = \Z/2[w_1,\dots,w_d].$$
Dual classes $\bar w_1, \bar w_2, \dots$ are defined by the equation
$$ (1+w_1 + \cdots + w_d)(1+ \bar w_1 + \bar w_2 + \cdots)=1,$$
and this allows one to write the classes $\bar w_k$ as polynomials in $w_1, \dots, w_d$.

The inclusion
\(\Gr_d(\R^{m}) \hookrightarrow \Gr_d(\R^\infty)\) then induces a
surjective ring homomorphism
\begin{align*}
  H^*(\Gr_d(\R^{\infty});\Z/{2}) &\to H^*(\Gr_d(\R^{m});\Z/{2})
\end{align*}
with kernel $J(d,m-d) = (\bar w_k \ | \ k>m-d)$. Now $\bar w_k$ can be interpreted as $w_k(\gamma_d^{\bot})$ where $\gamma_d^{\bot}$ is the $(m-d)$--dimensional bundle complimentary to $\gamma_d$.

We record some useful consequences. To state these, it is useful to let
$$i: \Gr_d(\R^{m-1}) \hookrightarrow \Gr_d(\R^m)$$ be the inclusion induced by the inclusion $\R^{m-1} \hra \R^m$, and to let
$$j: \Gr_{d-1}(\R^{m-1}) \hookrightarrow \Gr_d(\R^m)$$
be the inclusion sending $V \subset \R^{m-1}$ to $V \oplus \R \subset \R^m$.

\begin{lem} \label{w and bar w lemma} {\bf (a)} \ The ideal $J(d,m-d)$ is generated by the $d$ classes
$\bar{w}_{m-d+1}, \bar{w}_{m-d+2} ,\ldots, \bar{w}_{m}$. \\

\noindent{\bf (b)} \ In $H^*(\Gr_d(\R^{m});\Z/{2})$, $w_d \bar w_{m-d} = 0$. \\

\noindent{\bf (c)} \ $\ker\{i^*\} = (\bar w_{m-d}) \subset H^*(\Gr_d(\R^{m});\Z/{2}).$ \\

\noindent{\bf (d)} \ $\ker\{j^*\} = (w_{d}) \subset H^*(\Gr_d(\R^{m});\Z/{2}).$ \\
\end{lem}
\begin{proof}  Statement (a) follows from the recursive relations among the $\bar w_k$'s.  Statement (b) follows from the equation
$$ (1+w_1 + \cdots + w_d)(1+ \bar w_1 + \cdots + \bar w_{m-d}) = 1$$
which holds in $H^*(\Gr_d(\R^{m});\Z/{2})$.  Statement (c) follows from the fact that $J(d,m-1-d) = J(d,m)+ (\bar w_{m-d})$, and (d) follows from (c), noting that $j$ can be written as the composite
$$ \Gr_{d-1}(\R^{m-1}) \simeq \Gr_{m-d}(\R^{m-1}) \xra{i} \Gr_{m-d}(\R^{m}) \simeq \Gr_{d}(\R^{m}),$$
where the indicated homeomorphisms are given by taking complimentary subspaces (and, in cohomology, these maps swap $w_i$'s with $\bar w_j$'s).
\end{proof}

We end this subsection with a couple more facts about $H^*(\Gr_d(\R^{m});\Z/{2})$.

An additive basis for \(H^q(\Gr_d(\R^{m});\Z/{2})\) is given by the monomials
$$ \left \{w_1^{r_1}w_2^{r_2}\cdots w_{d}^{r_d} \ | \  \sum_{i=1}^d r_i \le m-d\right \},$$
so the top dimensional class is $w_d^{m-d}$ in degree $d(m-d)$.
See \cite{Jaworowski 89}.

The Wu formulae \cite[p.94]{milnor stasheff} are closed formulae for $Sq^iw_j$, and, in theory, formulae for $Q_n(w_j)$ follow.

\subsection{A description of the cofiber $C_d(\R^m)$ and its cohomology.} \label{cofib subsection}
\mbox{}

Recall that $C_d(\R^m)$ is defined as the cofiber of the inclusion $\Gr_d(\R^{m-1}) \xra{i} \Gr_d(\R^m)$.  This cofiber can be identified as a Thom space as follows.

\begin{prop} \label{cofib = thom space prop}  Let $S(\gamma_{d-1}^{\bot})$ and $D(\gamma_{d-1}^{\bot})$ be the sphere and disk bundles associated to $\gamma_{d-1}^{\bot} \ra \Gr_{d-1}(\R^{m-1})$.  There is a pushout
\begin{equation*}
\SelectTips{cm}{}
\xymatrix{
S(\gamma_{d-1}^{\bot}) \ar[d] \ar[r]^f & \Gr_d(\R^{m-1}) \ar[d]  \\
D(\gamma_{d-1}^{\bot}) \ar[r]^f & \Gr_d(\R^{m}), }
\end{equation*}
inducing a homeomorphism $f: Th(\gamma_{d-1}^{\bot}) \xra{\sim} C_d(\R^m)$, such that the composite $\Gr_{d-1}(\R^{m-1}) \xra[\sim]{\text{0-section}} D(\gamma_{d-1}^{\bot}) \xra{f} \Gr_d(\R^{m})$ is the map $j$ of \lemref{w and bar w lemma}.
\end{prop}
\begin{proof}  Recall that $D(\gamma_{d-1}^{\bot}) = \{(V,v) \ | \ V \in \Gr_{d-1}(\R^{m-1}), v \in V^{\bot}, |v| \leq 1 \}$, and $S(\gamma_{d-1}^{\bot}) = \{(V,v) \ | \ V \in \Gr_{d-1}(\R^{m-1}), v \in V^{\bot}, |v| = 1 \}$.

We define $f: D(\gamma_{d-1}^{\bot}) \ra \Gr_d(\R^{m})$ by the formula
$$f(V,v) = V + \langle v+ \sqrt{1-|v|^2}e_{m}\rangle,$$ where $e_m$ is the $m$th standard basis vector in $\R^m$.  We claim this $f$ has the needed properties.

Firstly, note that $f(V,\mathbf 0) = V + \langle e_m\rangle = V \oplus \R = j(V)$.

Secondly, $(V,v) \in S(\gamma_{d-1}^{\bot}) \Leftrightarrow f(V,v) = V + \langle v \rangle$, and so is an element of $\Gr_d(\R^{m-1})$.  Furthermore,
$f: S(\gamma_{d-1}^{\bot}) \ra  \Gr_d(\R^{m-1})$ is surjective: given any $W \in \Gr_d(\R^{m-1})$, if we choose any $(d-1)$-dimensional subspace $V$ of $W$, and a unit length vector $v \in W$ in the 1--dimensional orthogonal complement, then $f(V,v) = W$.

Finally, we need to check that $f$ is bijective on $\overset{\circ}{D}(\gamma_{d-1}^{\bot}) = D(\gamma_{d-1}^{\bot}) - S(\gamma_{d-1}^{\bot})$.  To check this,
let \(W \in
  \Gr_{d}(\R^{m})\) be a $d$--dimensional subspace of $\R^m$ not contained in $\R^{m-1}$, so that $V = W \cap \R^{m-1} \in \Gr_{d-1}(\R^{m-1})$.
Let \(V^\perp\) be the complement of \(V\) in of
    \(\R^{m}\) so that \(W \cap V^\perp\) is one-dimensional, and let $v$ be
    the unique unit vector \(v \in W \cap V^\perp\) such that \(v\) has
    positive \(m\)th coordinate. Let \(\pi \colon \R^m \to
    \R^{m-1}\) be the standard projection.  We
    claim that \(f(V,\pi(v))=W\) and $(V,\pi(v))$ is the unique point in $ \overset{\circ}{D}(\gamma_{d-1}^{\bot})$ with this property: since \(|v|=1\) the
    $m$th component of \(v\) is $\sqrt{1-|\pi(v)|^2}$, thus $v=\pi(v) + \sqrt{1-|\pi(v)|^2}e_m$ and so \(f(V,\pi(v))=V + \langle v
    \rangle = W\).
\end{proof}

Let $u_{\gamma_{d-1}^{\bot}}\in \widetilde H^{m-d}(Th(\gamma_{d-1}^{\bot}))$ be the Thom class of $\gamma_{d-1}^{\bot} \ra \Gr_{d-1}(\R^{m-1})$.   Then $\widetilde H^{m-d}(Th(\gamma_{*}^{\bot}))$ is a free rank 1 $H^*(\Gr_{d-1}(\R^{m-1}))$--module  on $u_{\gamma_{d-1}^{\bot}}$.  Meanwhile, $H^*(\Gr_{d}(\R^{m}))$ is a $H^*(\Gr_{d-1}(\R^{m-1}))$--module via $j^*$, and the ideal $(\bar w_{m-d}) = \widetilde H^*(C_d(\R^m))$ is a submodule.  The proposition thus implies the following.

\begin{cor} $f^*: \widetilde H^*(C_d(\R^m)) \xra{\sim} \widetilde H^{*}(Th(\gamma_{d-1}^{\bot}))$
is an isomorphism of free rank 1 $H^*(\Gr_{d-1}(\R^{m-1}))$--modules, and $f^*(\bar w_{m-d}) = u_{\gamma_{d-1}^{\bot}}$.
\end{cor}

\subsection{The characteristic class associated to $Q_n$ and a twisted $Q_n$--module.} \label{twisted subsection}

Let $\alpha_n \in H^{2^{n+1}-1}(Gr_{d-1}(\R^{m-1}))$ be defined as the element satisfying $Q_n(\bar w_{m-d}) = \alpha_n \bar w_{m-d} \in  \widetilde H^*(C_d(\R^m))$.  Then define
$$ \widehat Q_n: H^{*}(Gr_{d-1}(\R^{m-1}))\ra H^{*+2^{n+1}-1}(Gr_{d-1}(\R^{m-1}))$$
by the formula
$$ \widehat Q_n(x) = Q_n(x) + x\alpha_n.$$

\begin{prop} \label{twisted prop} $\widehat Q_n^2 = 0$, and the chain complex $(H^*(\Gr_{d-1}(\R^{m-1})),\widehat Q_n)$ is isomorphic to the chain complex $(\widetilde H^{*+m-d}(C_d(\R^m)), Q_n)$.
\end{prop}
\begin{proof} Let $\Theta: H^*(Gr_{d-1}(\R^{m-1})) \ra \widetilde H^{*+m-d}(C_d(\R^m))$ be the isomorphism established in the last subsection: $\Theta(x) = x\bar w_{m-d}$.  The proposition follows once we check that $\Theta(\widehat Q_n(x)) = Q_n(\Theta(x))$. We compute:
\begin{equation*}
\begin{split}
\Theta(\widehat Q_n(x))
& = \hat Q_n(x)\bar w_{m-d} \\
& = (Q_n(x)+ x\alpha_n)\bar w_{m-d} \\
& = Q_n(x)\bar w_{m-d} + x(\alpha_n \bar w_{m-d}) \\
& = Q_n(x)\bar w_{m-d} + xQ_n(\bar w_{m-d}) \\
& = Q_n(x\bar w_{m-d}) =  Q_n(\Theta(x)).
\end{split}
\end{equation*}
\end{proof}

It is useful to put the class $\alpha_n$ in context.  Given any element $a$ in the Steenrod algebra $\A$, one gets a characteristic class $w_a(\xi) \in H^{|a|}(B;\Z/2)$ associated to any real vector bundle $\xi \ra B$: $w_a(\xi)$ is defined as the element satisfying
$a(u_{\xi}) = w_a(\xi)u_{\xi} \in \widetilde H^{\dim \xi + |a|}(Th(\xi);\Z/2)$, where $u_\xi$ is the Thom class of $\xi$.   So, for example, $w_{Sq^n}(\xi) = w_n(\xi)$, and, relevant for us, our class $\alpha_n$ equals $w_{Q_n}(\xi)$ when $\xi = \gamma_{d-1}^{\bot} \ra \Gr_{d-1}(\R^{m-1})$.

We have the following characterization of $w_{Q_n}$.

\begin{prop} \label{Qn char class prop}  $w_{Q_n}$ is the unique characteristic class satisfying the following two properties:

{\bf (a)} \ $w_{Q_n}(\xi \oplus \nu)) = w_{Q_n}(\xi) + w_{Q_n}(\nu)$. \\

{\bf (b)} \ If $\gamma \ra B$ is one dimensional, then $w_{Q_n}(\gamma) = w_1(\gamma)^{2^{n+1}-1}$.
\end{prop}
\begin{proof} Property (a) follows from the fact that $Q_n$ is primitive in $\A$ (or, equivalently, that $Q_n$ acts a derivation).  To see Property (b), one first calculates that $Q_n(t) = t^{2^n+1} \in \Z/2[t] = H^*(\R P^{\infty};\Z/2)$, recalling that $Q_0 = Sq^1$, and $Q_n = Sq^{2^n}Q_{n-1} + Q_{n-1}Sq^{2^n}$.  Then Property (b) follows, since if $\gamma$ is the universal line bundle over $\R P^{\infty}$, then $u_{\gamma} = t$.  Uniqueness follows from the splitting principle.
\end{proof}

\begin{rem}  Thus $w_{Q_n}(\xi)$ agrees with the `s-class' $s_{2^{n+1}-1}(\xi)$, analogous to the class of the same name for complex vector bundles as defined in \cite[\S 16]{milnor stasheff}.  (These $s_I$'s are {\em not} the same as the $s_{\lambda}$ of the next subsection: these are two conflicting and standard usages.)
\end{rem}

Since $\gamma_{d-1}^{\bot} \oplus \gamma_{d-1}$ is trivial, property (b) has the following consequence.

\begin{cor} \label{Qn char class cor} $\alpha_n = w_{Q_n}(\gamma_{d-1}) \in H^{2^{n+1}-1}(\Gr_{d-1}(\R^{m-1}); \Z/2).$
\end{cor}

\subsection{The description of $H^*(\Gr_d(\R^m);\Z/2)$ via Schubert cells, and Lenart's formula.}    \label{lenart formula sec}
\mbox{}
\vspace{.1in}

For the purposes of proving \thmref{even m properties thm}, we use an alternative description of $H^*(\Gr_d(\R^m); \Z/2)$.

We recall the cell structure of \(\Gr_d(\R^{d+c})\)
as described in \cite[\S6]{milnor stasheff}. A
\emph{Schubert symbol} \(\lambda = (\lambda_1,\ldots,\lambda_d)\) of
\(\Gr_d(\R^{m})\) is a sequence of integers
  \[ m-d \ge \lambda_1 \ge \lambda_2 \ge \ldots \ge \lambda_d > 0.\] The \emph{weight}
  of \(\lambda\) is defined to be
  \(\sum \lambda_i\) and is denoted \(|\lambda|\).
Such a $\lambda$ is a partition contained inside of a \(d \times (m-d)\)
grid when depicted as Young diagrams:
diagrams with \(\lambda_i\) boxes in the
\(i\)th row.

To each such partition is associated a Schubert cell \(e(\lambda)\) of dimension $\lambda$ in \(\Gr_d(\R^{m})\) defined by
 \[e(\lambda) = \{V \in \Gr_d(\R^{m}) \mid \dim(V \cap
   \R^{i+\lambda_{d+1-i}})\ge i \text{ for } 1 \le i \le m\}.\]
This cell decomposition of the Grassmanian leads to the dual Schubert cell basis for $H^{*}(\Gr_d(\R^{m});\Z/{2})$ with basis elements
$s_{\lambda} \in H^{|\lambda|}(\Gr_d(\R^{m});\Z/{2})$.

With this notation, one has that $w_i = s_{(1^i)}$ and $\bar w_j = s_(j)$. Though we don't use this here, it is worth noting that the cohomology ring structure in this basis is described by the Littlewood-Richardson rule of symmetric function theory.

To state Lenart's formula for calculating
\(Q_n\) on a Schubert basis element \cite{lenart}, we need some combinatorial
definitions. Given a Young diagram \(\lambda\) that includes into
another Young diagram \(\mu\), one can form the complement
\(\mu / \lambda\). For example,
\begin{figure}[h]
  \centering
\ytableausetup{centertableaux,boxsize=1em}
\[
   \lambda = \begin{ytableau}
    \phantom{0} & &\\
    \phantom{0} \\
  \end{ytableau} \qquad
  \mu = \begin{ytableau}
    \phantom{0} & & &\\
    \phantom{0} & & \\
    \phantom{0}
  \end{ytableau} \qquad
  \mu / \lambda = \begin{ytableau}
    \none & \none & \none &\\
    \none & & \\
    \phantom{0}
  \end{ytableau}.
\]
 \end{figure}

The \emph{content} of a box \(b\) of \(\mu\) in row \(i\) and column
\(j\) is defined to be \(\content(b)=j-i\). For a box \(b\) in the
skew shape \(\mu / \lambda\), we define its content to be the content
of \(b\) embedded in \(\mu\). Here we fill in the contents of the
diagrams from above
  \begin{figure}[h]
  \centering
\ytableausetup{centertableaux,boxsize=1.5em}
\[
   \lambda = \begin{ytableau}
    0 & 1 & 2\\
    -1 \\
  \end{ytableau} \qquad
  \mu = \begin{ytableau}
    0 & 1 & 2 & 3\\
    -1 & 0 & 1 \\
    -2
  \end{ytableau} \qquad
  \mu / \lambda = \begin{ytableau}
    \none & \none & \none & 3\\
    \none & 0 &  1\\
    -2
  \end{ytableau}
\]
 \end{figure}

A skew-shape is said to be \emph{connected} when each pair of boxes in
the diagram is connected by a sequence of boxes that each share an
edge. A shape \(\lambda\) is called a \emph{border strip}, if it is
connected and does not contain a \(2 \times 2\) block of boxes. A
shape satisfying just the second condition is called a \emph{broken
  border} strip, and in particular, a border strip is an example of a
broken border strip with just one connected component. If \(\lambda\)
is a broken border strip, then we denote by
\(\connectedcomponents(\lambda)\) the number of connected components
of \(\lambda\). If \(\lambda\) is not a broken border strip, then we
define \(\connectedcomponents(\lambda) = \infty\). For example, in the
next diagram, \(\lambda_1\) is a border strip, \(\lambda_2\) is a
broken border strip that is not a border strip, and \(\lambda_3\) is
an example of a shape that is neither.
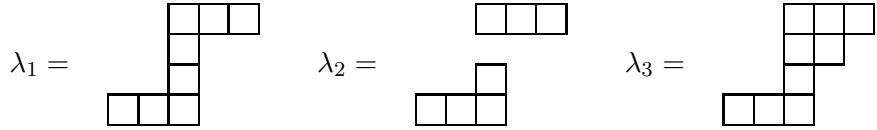
\begin{figure}[h]
  \centering
\ytableausetup{centertableaux,boxsize=1em}
\[\lambda_1 = \begin{ytableau}
    \none & \none & \none &  & &\\
    \none & \none & \none &\\
    \none & \none & \none &\\
    \none &  & & \\
  \end{ytableau}
  \qquad
  \lambda_2 =\begin{ytableau}
    \none & \none & \none &  & &\\
    \none & \none & \none & \none \\
    \none & \none & \none &  \\
    \none &  & & \\
  \end{ytableau} \qquad
   \lambda_3 = \begin{ytableau}
    \none & \none & \none &  & &\\
    \none & \none & \none & & \\
    \none & \none & \none &  \\
    \none &  & & \\
  \end{ytableau}
\]
\end{figure}

A \emph{sharp corner} of a broken border strip is a box with no north,
no west and no northwest neighbors. A \emph{dull corner} is a box with
both north and west neighbors, but no northwest neighbor. Let
\(\corners(\mu/\lambda)\) denote the set of sharp and dull corners of
\(\mu/\lambda\). For example, in the following diagram the sharp
corners have been labeled \(S\) and the dull corners have been labeled
\(D\).
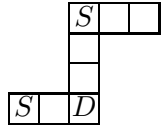
\begin{figure}[h]
  \centering
\ytableausetup{centertableaux,boxsize=1em}
\[\begin{ytableau}
    \none & \none & \none & S & &\\
    \none & \none & \none &\\
    \none & \none & \none &\\
    \none & S & & D\\
\end{ytableau} \]
 \end{figure}

We are now ready to state Lenart's formula from \cite{lenart}:
  \begin{align}
  \label{formula:Dk}
    Q_{n} (s_\lambda) &= \sum_{\substack{\mu \supset \lambda \colon |\mu|-|\lambda| = {2^{n+1}-1}\\ cc(\mu / \lambda )\le 2}} d_{\lambda\mu} s_\mu,
\end{align}
where \(\mu/\lambda\) must be a broken border strip and
\begin{align}
\label{formula:lenart coefficient}
  d_{\lambda \mu} &=
    \begin{cases}\displaystyle
      \sum_{b \in C(\mu/\lambda)} c(b) & \mu/\lambda \text{ is connected}\\
      1 & \mu/\lambda \text{ is disconnected.}
    \end{cases}
\end{align}
\begin{ex}
   \ytableausetup{centertableaux,boxsize=.5em} As an example we compute
  \(Q_1\) on \(w_1 = s_{\ydiagram{1}}\) in the Schubert basis in
  \(\Gr_2(\R^6)\). There are three basis elements in degree four
  \[\mu_1 = \ydiagram{4} \qquad \mu_2 = \ydiagram{3,1} \qquad \mu_3 = \ydiagram{2,2}.\]
  To compute \(Q_n(s_{\ydiagram{1}})\) using (\ref{formula:Dk}) we
  must consider each complement. Let \(\lambda = \ydiagram{1}\). For
  \(\mu_1\) we have \ytableausetup{centertableaux,boxsize=1.5em}
  \[\mu_1 / \lambda =\ytableaushort{0123} \ / \
    \ydiagram{1}=\ytableaushort{123}.\] The complement is a border
  strip and there is just one sharp corner (the left most corner) and
  no dull corners. The content of the sharp corner is \(1\) modulo
  two, hence \(d_{\lambda \mu_1}=1\), and so \(s_{\mu_1}\) is in the
  expansion of \(Q_1(s_{\lambda})\). Next we consider
  \[\mu_2 / \lambda = \ytableaushort{012,{-1}} \ / \ \ydiagram{1} =
    \ytableaushort{\none 12,{-1}}.\] This is a disconnected broken
  border strip, hence \(d_{\lambda \mu_2}=1\), and so \(s_{\mu_2}\) is
  in the expansion. Finally,
  \[\mu_3 / \lambda = \ytableaushort{01,{-1}{-2}} \ / \ \ydiagram{1} =
    \ytableaushort{\none 1,{-1}{-2}}.\] There are two sharp corners,
  one of content \(-1\) and the other of content \(1\). There is also
  one dull corner of content \(-2\). This means
  \(d_{\lambda \mu_3} = (-1) + 1 + 2 \equiv 0\), and so \(s_{\mu_3}\)
  is not in the expansion. Hence,
  \[\ytableausetup{centertableaux,boxsize=.5em}
    Q_1\left(s_{\ydiagram{1}}\right) =  s_{\ydiagram{4}} + s_{\ydiagram{3,1}}.
  \]
\end{ex}

\section{Results about $H^*(\Gr_d(\R^m);Q_n)$ when $m$ is even}
\label{m even section}

\begin{proof}[Proof of \thmref{even m properties thm}(a)]
  We are going to show that \(Q_n (s_\lambda) = 0\) for each Schubert
  basis element \(s_\lambda\) in degree \(d(m-d)-2^{n+1}+1\).  Since
  \(s_{(d^{(m-d)})}\) is the only class in degree \(d(m-d)\),
  \(Q_n(s_\lambda) = d_{\lambda {{(d^{(m-d)})}}} {s_{(d^{(m-d)})}}\),
  where \(d_{\lambda {{(d^{(m-d)})}}}\) is given by
  (\ref{formula:lenart coefficient}). We must only consider
  \(\lambda\) such that \({{(d^{(m-d)})}}/\lambda\) is a broken border
  strip. As \({{(d^{(m-d)})}}\) is a \(d\times (m-d)\) grid the
  complement \({(d^{(m-d)})}/\lambda\) is always connected and so if
  \({{(d^{(m-d)})}} /\lambda\) is a broken border strip it must be, in
  particular, a border strip. If \({{(d^{(m-d)})}} /\lambda\) is a
  border strip, then it must be one of three types: \begin{enumerate}
 \item \({{(d^{(m-d)})}}/\lambda\) is the last row of \({(d^{(m-d)})}\),
 \item \({{(d^{(m-d)})}}/\lambda\) is the last column of \({(d^{(m-d)})}\),
 \item \({{(d^{(m-d)})}}/\lambda\) is the union of the last row and last column of \({(d^{(m-d)})}\).
 \end{enumerate}
 We will show that \(d_{\lambda{{(d^{(m-d)})}}}= 0\) in each of these cases. As
 \(m\) was assumed to be even, the content of the right most bottom
 box of \({{(d^{(m-d)})}}\) is also even.

\begin{enumerate}
\item For the first case, there is just one sharp corner, namely the
  left most box, and there are no dull corners. Since the strip is of
  odd length, namely, \(2^{n+1}-1\), the left most box and the right
  most box have the same content modulo two. Hence, the content of this
  sharp corner is zero modulo two, and so \(d_{\lambda{{(d^{(m-d)})}}}=0\).
\item For the second case, the argument is exactly the same, but with
  the sharp corner on the top.
\item For the third case, the content of the sharp corner on the
  bottom left and the content of the sharp corner on the
  top right agree modulo two, because the border strip is of odd
  length. There is one dull corner in the bottom right and it is zero
  modulo two. Thus, the two sharp corners cancel and the dull corner
  contributes nothing.
\end{enumerate}
Thus, in all cases \(Q_n(s_\lambda) = 0\) for
\(s_\lambda\) in degree \(m(m-d)-2^{n+1}+1\). This completes the proof
that the top class is not in the image of \(Q_n\) for even \(m\).
\end{proof}

\begin{proof}[Proof of \thmref{even m properties thm}(b)]

We wish to prove that, when $m$ is even, then the chain complexes $(\widetilde H^*(C_d(\R^{m}));Q_n)$ and $(H^{d(m-d)-*}(Gr_{d-1}(\R^{m-1}));Q_n)$ are dual.

By \propref{twisted prop}, the $(\widetilde H^{*+m-d}(C_d(\R^{m}));Q_n)$ is isomorphic to the chain complex $( H^*(\Gr_{d-1}(\R^{m-1}));\widehat Q_n)$, where we recall that $\widehat Q_n(y) = Q_n(y) + y \alpha_n$, and that $\alpha_n \bar w_{m-d} = Q_n(\bar w_{m-d}) \in H^*(\Gr_d(\R^m);\Z/2)$.

So we need to check that the chain complexes $H^*(\Gr_{d-1}(\R^{m-1}));\widehat Q_n)$ and $(H^{(d-1)(m-d)-*}(Gr_{d-1}(\R^{m-1}));Q_n)$ are dual.  This means we need to show that, whenever $x,y \in H^*(Gr_{d-1}(\R^{m-1});\Z/2)$ satisfy $|x| + |y| + |Q_n| = (d-1)(m-1)$, then
$$ Q_n(x) y  = x \widehat Q_n(y).$$

By \thmref{even m properties thm}(a), we know that $Q_n(xy\bar w_{m-d}) = 0 \in H^{d(m-d)}(\Gr_d(\R^m);\Z/2)$.  Thus, in $H^{d(m-d)}(\Gr_d(\R^m);\Z/2)$, we have
\begin{equation*}
\begin{split}
0 & = Q_n(xy\bar w_{m-d}) \\
& = Q_n(x)y\bar w_{m-d} + x Q_n(y) \bar w_{m-d} + xy Q_n(\bar w_{m-d}) \\
& = Q_n(x)y\bar w_{m-d} + x Q_n(y) \bar w_{m-d} + xy \alpha_n \bar w_{m-d} \\
& = (Q_n(x)y + x Q_n(y) + xy \alpha_n) \bar w_{m-d} \\
& = (Q_n(x)y + x \widehat Q_n(y)) \bar w_{m-d},
\end{split}
\end{equation*}
and we conclude that $0 =  Q_n(x) y + x \widehat Q_n(y) \in H^{(d-1)(m-d)}(Gr_{d-1}(\R^{m-1});\Z/2)$.
\end{proof}

\begin{proof}[Proof of \thmref{even m properties thm}(c)]
Recall that $k_{Q_n}(X)$ denotes the rank of the $Q_n$--homology $H^*(X;Q_n)$.  Similarly, let $\bar k_{Q_n}(X)$ denote the rank of $\widetilde H^*(X;Q_n)$.

Let $m = 2^{n+1} - \epsilon + 2l$ with $\epsilon =$ 0 or 1, and $l \geq 0$.  Let
$$ k^G_n(d,m) = \sum_{i=0}^{\left\lfloor d/2 \right\rfloor} \binom{2^{n+1}-\epsilon}{d-2i}\binom{l}{i}.$$

We start with the first part of \thmref{even m properties thm}(c). This asserts that, when $m$ is even,  if we assume that $$k_{Q_n}(\Gr_{d}(\R^{m-1})) = k^G_n(d,m-1) \text{ and } k_{Q_n}(\Gr_{d-1}(\R^{m-1})) = k^G_n(d-1,m-1),$$ then we can conclude that $k_{Q_n}(\Gr_{d}(\R^{m})) = k^G_n(d,m)$.

\thmref{lower bound thm} tells us that $k_{Q_n}(\Gr_{d}(\R^{m})) \geq k^G_n(d,m)$.

Since we have a short exact sequence
$$ 0 \ra \widetilde H^*(C_d(\R^m)) \ra H^*(\Gr_d(\R^m))) \ra H^*(\Gr_d(\R^{m-1})) \ra 0,$$ we see that
$$k_{Q_n}(\Gr_{d}(\R^{m-1})) + \bar k_{Q_n}(C_d(\R^m)) \geq k_{Q_n}(\Gr_{d}(\R^{m})),$$
with equality if and only if the associated long exact $Q_n$--homology sequence is still short exact.

Since $m$ is even, \thmref{even m properties thm}(b) applies, and tells us that
$\bar k_{Q_n}(C_d(\R^m)) = k_{Q_n}(\Gr_{d-1}(\R^{m-1}))$.

Putting this all together, under our assumptions, we have that
$$ k^G_n(d,m-1) + k^G_n(d-1,m-1) \geq k_{Q_n}(\Gr_{d}(\R^{m})) \geq k^G_n(d,m).$$
That these would be, in fact, equalities, follows from the next lemma.

\begin{lem} If $m=2^{n+1} + 2l$ with $l\geq 0$, then
$$ k^G_n(d,m-1) + k^G_n(d-1,m-1) = k^G_n(d,m).$$
\end{lem}
\begin{proof}
\begin{equation*}
\begin{split}
k^G_n(d,m-1) + k^G_n(d-1,m-1) &
= \sum_i \left [ \binom{2^{n+1}-1}{d-2i} + \binom{2^{n+1}-1}{d-2i-1}\right ]\binom{l}{i} \\
  & =\sum_i \binom{2^{n+1}}{d-2i}\binom{l}{i} \\
  & = k^G_n(d,m).
\end{split}
\end{equation*}
\end{proof}

When all of this happens, we then see that the $Q_n$--homology long exact sequence really is still short exact, and also that the $K(n)$--AHSS must collapse for these three spaces.  Thus there is also a short exact sequence
$$ 0 \ra \widetilde K(n)^*(C_d(\R^m)) \xra{p^*} K(n)^*(\Gr_d(\R^m)) \xra{i^*} K(n)^*(\Gr_d(\R^{m-1})) \ra 0. $$

Finally, the top cohomology class in $H^{d(m-d)}(\Gr_d(\R^m);\Z/2)$ will be a permanent cycle in the AHSS computing $K(n)^*(\Gr_d(\R^m))$ and thus also in the AHSS computing $k(n)^*(\Gr_d(\R^m))$, and this is equivalent to saying that $\Gr_d(\R^m)$ is $k(n)$--oriented.
\end{proof}

\section{Results about $H^*(\Gr_d(\R^m);Q_n)$ when $d=2$} \label{d=2 section}
\mbox{}

In this section we present our results about the $Q_n$--homology of $\Gr_2(\R^m)$, with the focus on understanding the case when $m$ has the form $2^{n+1}-1+ 2l$. \\

To begin with, we know the following:
\begin{itemize}
\item $H^*(\Gr_2(\R^m);\Z/2) = \Z/2[w_1,w_2]/(\bar w_{m-1}, \bar w_m)$.
\item In $H^*(\Gr_2(\R^m);\Z/2)$, the ideal $\widetilde H^*(C_2(\R^m); \Z/2)$ has an additive basis $\{ w_1^i \bar w_{m-2} \ | \ 0 \leq i \leq m-2\}$.
\end{itemize}

Now we collect results that hold in $H^*(\Gr_2(\R^{\infty}); \Z/2)$.

\begin{lem} \label{bar w lem} In $H^*(\Gr_2(\R^{\infty}); \Z/2)$ we have the following.

(a) $\bar w_0 = 1$, $\bar w_1 = w_1$, and, recursively,
$ \bar w_k = w_1 \bar w_{k-1} + w_2 \bar w_{k-2}$.

(b) $\displaystyle w_2^j\bar w_k = \sum_i \binom{j}{i}w_1^{j-i}\bar w_{k+j+i}$

(c) $\displaystyle \bar w_k = \sum_j \binom{k-j}{j} w_1^{k-2j}w_2^j$.

(d) $\bar w_{2^b-1} = w_1^{2^b-1}$ for all $b \geq 0$.

(e) $\displaystyle \bar w_{2^b-2} = \sum_{c=0}^{b-1} w_1^{2^b-2^{c+1}}w_2^{2^c-1}$ for all $b \geq 1$.
\end{lem}
\begin{proof}

The homogeneous components of the equation $0 = (1+w_1+w_2)(1+\bar w_1 + \bar w_2 + \dots)$ give statement (a).

Statement (b) is proved by induction on $j$.  The case when $j=0$ is trivial, and statement (a) rewrites as $w_2 \bar w_k = w_1 \bar w_{k+1} + \bar w_{k+2}$, which is the case when $j=1$. One then computes
\begin{equation*}
\begin{split}
w_2^j\bar w_k &
= w_2(w_2^{j-1}\bar w_k) \\
  & = \sum_i \binom{j-1}{i}w_1^{j-1-i}w_2\bar w_{k+j-1+i} \\
  & = \sum_i \binom{j-1}{i}[w_1^{j-i}\bar w_{k+j+i}+ w_1^{j-1-i}\bar w_{k+j+i+1}] \\
  & = \sum_i \left [\binom{j-1}{i}+\binom{j-1}{i-1}\right ] w_1^{j-i}\bar w_{k+j+i} \\
  & = \sum_i \binom{j}{i}w_1^{j-i}\bar w_{k+j+i}.
\end{split}
\end{equation*}

For (c), note that $\bar w_k$ is the homogeneous component of degree $k$ in $\bar w = (1 + w_1 + w_2)^{-1} = \sum_{t=0}^{\infty} (w_1+w_2)^t$.

Statement (d) follows from (c):
$$\bar w_{2^b-1} = \sum_j \binom{2^b-1-j}{j} w_1^{k-2j}w_2^j  = w_1^{2^b-1},$$
using that $\binom{2^b-1-j}{j} \equiv 1 \mod 2$ only if $j=0$.

Similarly, statement (e) follows from (c):
$$\bar w_{2^b-2} = \sum_j \binom{2^b-2-j}{j} w_1^{k-2j}w_2^j  = \sum_{c=0}^{b-1} w_1^{2^b-2^{c+1}}w_2^{2^c-1},$$
using that $\binom{2^b-2-j}{j} \equiv 1 \mod 2$ if and only if $j=2^c-1$ with $0\leq j \leq b-1$.
\end{proof}

Now we determine the action of $Q_n$ on various classes.

\begin{lem} \label{Qn w1 lem}  In $H^*(\Gr_2(\R^{\infty});\Z/2)$, $Q_n(w_1) = w_1^{2^{n+1}} = w_1 \bar w_{2^{n+1}-1}$.
\end{lem}
\begin{proof} The first equation here was already noted in the proof of \propref{Qn char class prop}, and the second follows from \lemref{bar w lem}(d).
\end{proof}

\begin{lem} \label{Qn w2 lem}  In $H^*(\Gr_2(\R^{\infty});\Z/2)$,
$$Q_n(w_2) =  \sum_{c=0}^{n} w_1^{2^{n+1}-2^{c+1}+1}w_2^{2^c}= w_1 w_2\bar w_{2^{n+1}-2}.$$
\end{lem}
\begin{proof}  The second equation here follows from \lemref{bar w lem}(e), so we just need to check the first.  We do this by induction on $n$, where the $n=0$ case is the easily checked:  $Q_0(w_2) = Sq^1(w_2) = w_1w_2$.

Before proceeding with the inductive step, we make two observations.

The first is that for $n \geq 1$, $Q_n(w_2) = Sq^{2^n}Q_{n-1}(w_2)$ because the other term, $Q_{n-1}Sq^{2^n}(w_2)$, will be zero.  This is clear if $n\geq 2$ as then $Sq^{2^n}(w_2) = 0$, and when $n=1$, we observe that $Q_0Sq^2(w_2) = Sq^1(w_2^2) = 0$.

The second observation is that $Sq(w_2) = w_2(1+w_1+w_2)$, so $Sq(w_2^{2^c}) = w_2^{2^c}(1+w_1^{2^c}+w_2^{2^c})$, and thus
$$Sq^j(w_2^{2^c}) = \begin{cases} w_2^{2^c} & \text{if } j=0 \\
w_1^{2^c}w_2^{2^c}& \text{if } j = 2^c \\ w_2^{2^{c+1}} & \text{if } j=2^{c+1} \\ 0 & \text{otherwise.}
\end{cases}$$

Now we check the inductive step of our proof.
\begin{equation*}
\begin{split}
Q_n(w_2) &
= Sq^{2^n}Q_{n-1}(w_2) \\
  & = \sum_{c=0}^{n-1} Sq^{2^n}(w_1^{2^n-2^{c+1}+1}w_2^{2^c}) \\
  & = \sum_{c=0}^{n-1}\sum_{j} Sq^{2^n-j}(w_1^{2^n-2^{c+1}+1})Sq^j(w_2^{2^c}).
\end{split}
\end{equation*}
The terms with $j=0$ are all zero, as $Sq^{2^n}(w_1^{2^n-2^{c+1}+1})=0$ by the unstable condition.  Similarly, the only nonzero term with $j=2^c$ is the term $w_1^{2^{n+1}-1}w_2$, when $c=0$.  Finally, one gets $w_1^{2^{n+1}-2^{c+2}+1}w_2^{2^{c+1}}$ when $j=2^{c+1}$ for all $0 \leq c \leq n-1$.
\end{proof}

We now turn our attention to the behavior of $Q_n$ on $\widetilde H^*(C_2(\R^m); \Z/2)$.

\begin{lem}  In $\widetilde H^*(C_2(\R^m); \Z/2)$, $Q_n(\bar w_{m-2}) = w_1^{2^{n+1}-1}\bar w_{m-2}$.
\end{lem}
\begin{proof}  By \corref{Qn char class cor}, $Q_n(\bar w_{m-2}) = w_{Q_n}(\gamma_1)\bar w_{m-2}$, where $\gamma_1 \ra \Gr_1(\R^{m-1})$ is the canonical line bundle, and \propref{Qn char class prop} tells us that $w_{Q_n}(\gamma_1)=w_1^{2^{n+1}-1}$.
\end{proof}

\begin{rem}  This lemma also admits a proof using the Schubert cell perspective.  See \cite[Lemma 4.9.13]{lloyd thesis}.
\end{rem}

As $Q_n$ is a derivation, the lemma, together with the calculation $Q_n(w_1) = w_1^{2^{n+1}}$, allows one to easily compute the $Q_n$--homology of $C_2(\R^m)$.  What results is the following.

\begin{prop} \label{cofiber Qn prop} (a) In $\widetilde H^*(C_2(\R^m);\Z/2)$,
$$Q_n(w_1^i\bar w_{m-2}) = \begin{cases}
w_1^{2^{n+1}-1+i}\bar w_{m-2} & \text{if } i \text{ is even} \\ 0 & \text{if } i \text{ is odd}.
\end{cases}$$

(b) If $m\leq 2^{n+1}$, then $Q_n$ acts as zero on $\widetilde H^*(C_2(\R^m);\Z/2)$. Thus $\tilde k_{Q_n}(C_2(\R^m)) = m-1$.

(c) If $m > 2^{n+1}$ and is even, then the classes $\{w_1^{2j-1}\bar w_{m-2} \ | \ 1 \leq j \leq 2^n-1\}$ and $\{w_1^{m-2j}\bar w_{m-2} \ | \ 1 \leq j \leq 2^n\}$ represent the $Q_n$--homology classes.  Thus $\tilde k_{Q_n}(C_2(\R^m)) = 2^{n+1}-1$.

(d) If $m > 2^{n+1}$ and is odd, then the classes $\{w_1^{2j-1}\bar w_{m-2} \ | \ 1 \leq j \leq 2^n-1\}$ and $\{w_1^{m-1-2j}\bar w_{m-2} \ | \ 1 \leq j \leq 2^n-1\}$ represent the $Q_n$--homology classes.  Thus $\tilde k_{Q_n}(C_2(\R^m)) = 2^{n+1}-2$.
\end{prop}

\begin{proof}[Proof of \thmref{d=2 odd m boundary thm}]  Let $m = 2^{n+1}+1+2l$.  We need to prove that the map
$$\widetilde H^{*}(C_2(\R^{m});Q_n) \xra{p^*} H^{*}(\Gr_2(\R^{m});Q_n)$$
is zero.  In other words, we need to show that representatives of the $Q_n$--homology classes in $\widetilde H^*(C_2(\R^m);\Z/2)$ are in the image of $Q_n$ when regarded in $H^*(\Gr_2(\R^m);\Z/2)$.

By \propref{cofiber Qn prop}(d), these representatives are in two families:
$$w_1^{1+2j}\bar w_{2^{n+1}-1+2l} \text{\hspace{.2in} and \hspace{.2in}} w_1^{2l+2+2j}\bar w_{2^{n+1}-1+2l},$$ both with $0 \leq j \leq 2^n-2$.

If we can find $a,b \in H^*(\Gr_2(\R^m;\Z/2)$ such that $Q_n(a) = w_1\bar w_{2^{n+1}-1+2l}$ and $Q_n(b) = w_1^{2l+2}\bar w_{2^{n+1}-1+2l}$, we will be done, as then
$$Q_n(w_1^{2j}a) = w_1^{1+2j}\bar w_{2^{n+1}-1+2l}\text{\hspace{.2in} and \hspace{.2in}} Q_n(w_1^{2j}b) = w_1^{2l+2+2j}\bar w_{2^{n+1}-1+2l}.$$

Thus the next two propositions finish the proof.
\end{proof}

\begin{prop} \label{a prop} In $H^*(\Gr_2(\R^{\infty});\Z/2)$,
$$Q_n(w_1\bar w_{2l}) = w_1\bar w_{2^{n+1}-1+2l}.$$
\end{prop}

\begin{prop} \label{b prop} In $H^*(\Gr_2(\R^{2^{n+1}+1+2l});\Z/2)$,
$$Q_n(w_2^{2l+1}) = w_1^{2l+2}\bar w_{2^{n+1}-1+2l}.$$
\end{prop}

Before proving these, we first run through how \thmref{d=2 odd m boundary thm} leads to the proof of \thmref{d=2 thm}.

\begin{proof}[Proof of \thmref{d=2 thm}]  Our goal is to show that if $m = 2^{n+1} - \epsilon + 2l$ with $\epsilon =$ 0 or 1, and $l \geq 0$, then
$k_{Q_n}(\Gr_2(\R^m)) = \binom{2^{n+1}-\epsilon}{2} + l$, the lower bound coming from \thmref{lower bound thm}.

We prove this by induction on $m$, with the two cases when $l=0$ already covered by \thmref{collapse thm}.
The case when $m$ is even is covered by \thmref{even m properties thm}(c), as we know our calculations are right for $(n,1,m-1)$, and by induction we can assume the theorem for $(n,2,m-1)$.

Suppose $m$ is odd, so $\epsilon = 1$ and $m-1 = 2^{n+1} + 2(l-1)$. By induction, we can assume that $k_{Q_n}(\Gr_2(\R^{m-1})) = \binom{2^{n+1}}{2}+ (l-1)$.  Then we have
\begin{equation*}
\begin{split}
k_{Q_n}(\Gr_2(\R^m)) &
= k_{Q_n}(\Gr_2(\R^{m-1})) - \bar k_{Q_n}(C_2(\R^m)) \text{ \hspace{.1in} (by \thmref{d=2 odd m boundary thm})}\\
  & = \binom{2^{n+1}}{2}+ (l-1) - (2^{n+1}-2) \text{ \hspace{.1in} (by \propref{cofiber Qn prop}(d))} \\
  & = \binom{2^{n+1}-1}{2} + l.
\end{split}
\end{equation*}
\end{proof}

It remains to prove Propositions \ref{a prop} and \ref{b prop}.

\begin{proof}[Proof of \propref{a prop}]
We prove by induction on $l$ that
$$Q_n(w_1\bar w_{2l}) = w_1\bar w_{2^{n+1}-1+2l}$$
holds in $H^*(\Gr_2(\R^{\infty});\Z/2)$.

We start the induction by checking both the $l=0$ and $l=1$ cases.

When $l=0$, this reads $Q_n(w_1) = w_1\bar w_{2^{n+1}}$, proved in \lemref{Qn w1 lem}.

We check the $l=1$ case using both \lemref{Qn w1 lem} and \lemref{Qn w2 lem}:
\begin{equation*}
\begin{split}
Q_n(w_1\bar w_2) &
= Q_n(w_1(w_2+w_1^2)) = Q_n(w_1w_2 + w_1^3) \\
  & = Q_n(w_1)w_2 + w_1Q_n(w_2) + w_1^2Q_n(w_1) \\
  & = w_1w_2\bar w_{2^{n+1}-1} + w_1^2 w_2\bar w_{2^{n+1}-2} + w_1^3\bar w_{2^{n+1}-1} \\
  & = w_1[w_2\bar w_{2^{n+1}-1} + w_1(w_2\bar w_{2^{n+1}-2} + w_1\bar w_{2^{n+1}-1})] \\
  & = w_1[w_2\bar w_{2^{n+1}-1} + w_1 \bar w_{2^{n+1}}] = w_1 \bar w_{2^{n+1}+1}.
\end{split}
\end{equation*}

For the inductive case, we use the identity $\bar w_k = w_2^2\bar w_{k-4}+ w_1^2\bar w_{k-2}$ which holds for all $k \geq 4$.  Then we have
\begin{equation*}
\begin{split}
Q_n(w_1\bar w_{2l}) &
= Q_n(w_1w_2^2\bar w_{2(l-2)} + w_1^3 \bar w_{2(l-1)})  \\
  & = w_2^2Q_n(w_1\bar w_{2(l-2)}) + w_1^2 Q_n(w_1 \bar w_{2(l-1)}) \\
  & = w_1w_2^2\bar w_{2^{n+1} + 2(l-2) -1} + w_1^3\bar w_{2^{n+1} + 2(l-1) -1} \\
  & = w_1[w_2^2\bar w_{2^{n+1} + 2(l-2) -1} + w_1^2\bar w_{2^{n+1} + 2(l-1) -1}] \\
  & = w_1 \bar w_{2^{n+1} + 2l -1}.
\end{split}
\end{equation*}

\end{proof}

\begin{proof}[Proof of \propref{b prop}]  We wish to prove that
$$Q_n(w_2^{2l+1}) = w_1^{2l+2}\bar w_{2^{n+1}-1+2l}$$
holds in $H^*(\Gr_2(\R^{2^{n+1}+1+2l});\Z/2)$.

We begin with a calculation in $H^*(\Gr_2(\R^{\infty});\Z/2)$:
\begin{equation*}
\begin{split}
Q_n(w_2^{2l+1}) & = w_2^{2l}Q_n(w_2) \\
  & = w_1w_2^{2l+1} \bar w_{2^{n+1}-2} \text{\hspace{.1in} (using \lemref{Qn w2 lem})}\\
  & = \sum_i \binom{2l+1}{i}w_1^{2l+2-i}\bar w_{2^{n+1}+2l-1+i} \text{\hspace{.1in} (using \lemref{bar w lem}(b))}.
\end{split}
\end{equation*}
When we project this sum onto
$$H^*(\Gr_2(\R^{2^{n+1}+1+2l});\Z/2) = \Z/2[w_1,w_2]/(\bar w_k \ | \ k\geq 2^{n+1}+2l),$$ only the term with $i=0$ is not zero. In other words
$$Q_n(w_2^{2l+1}) = w_1^{2l+2}\bar w_{2^{n+1}-1+2l}$$
holds in $H^*(\Gr_2(\R^{2^{n+1}+1+2l});\Z/2)$.
\end{proof}

\section{Towards the conjectures} \label{speculation section}
\mbox{}
\vspace{.1in}

As organized in this paper, we are trying to calculate $H^*(\Gr_d(\R^m); Q_n)$ by induction on $m$ (and $d$) with two steps:
\begin{itemize}
\item  calculate $\widetilde H^*(C_d(\R^m);Q_n)$, recalling that $C_d(\R^m)$ is the Thom space of a bundle over $\Gr_{d-1}(\R^{m-1})$, and
\item  calculate $\delta: H^*(\Gr_d(\R^{m-1});Q_n) \ra \widetilde H^{*+2^{n+1}-1}(C_d(\R^{m});Q_n)$.
\end{itemize}

When $m$ is even, \thmref{even m properties thm} says we can carry through with this plan.  In this section we speculate about how things might go when $m$ is odd.

Firstly, we have the analogues of \thmref{collapse thm} and \thmref{lower bound thm} for $\bar k_n(C_d(\R^m)) = \dim_{K(n)_*}\widetilde K(n)^*(C_d(\R^m))$.

\begin{thm} \label{C_d collapse thm}  If $m \leq 2^{n+1}$, then $\bar k_n(C_d(\R^m)) = \binom{m-1}{d-1}$.
\end{thm}

\begin{proof} \thmref{collapse thm} implies that, if $m \leq 2^{n+1}$, the inclusion $\Gr_d(\R^{m-1}) \ra \Gr_d(\R^m)$ induces an inclusion $K(n)_*(\Gr_d(\R^{m-1})) \ra K(n)_*(\Gr_d(\R^m))$, as this is true in mod $p$ homology.  Thus
\begin{equation*}
\begin{split}
\bar k_n(C_d(\R^m)) &
= k_n(\Gr_d(\R^m)) -  k_n(\Gr_d(\R^{m-1})) \\
  & = \binom{m}{d} - \binom{m-1}{d} = \binom{m-1}{d-1}.
\end{split}
\end{equation*}
\end{proof}

\begin{thm} \label{Cd lower bound thm} Let $m = 2^{n+1} - \epsilon + 2l$ with $\epsilon =$ 0 or 1, and $l \geq 0$.  Then
$$ k_n(C_d(\R^m)) \geq \sum_{i=0}^{\left\lfloor d/2 \right\rfloor} \binom{2^{n+1}- 1-\epsilon}{d-1-2i}\binom{l}{i}.$$
\end{thm}

\begin{proof}  The proof is similar to the proof of \thmref{lower bound thm}, with a little tweak.

If $V$ is a real representation of $C_4$, and $W$ is a subrepresentation, let $C_d(V,W)$ denote the cofiber of the inclusion $\Gr_d(W) \hra \Gr_d(V)$: this is a based $C_4$ space.

If $\dim V = m$ and $\dim W = m-1$, then $C_d(\R^m) = C_d(V,W)$ and thus $\bar k_n(C_d(\R^m)) \geq \bar k_n(C_d(V,W)^{C_4})$, by our chromatic fixed point theorem, \thmref{chromatic fixed point thm}.  Furthermore, $C_d(V,W)^{C_4}$ will be the cofiber of the inclusion $\Gr_d(W)^{C_4} \hra \Gr_d(V)^{C_4}$.

Now we choose $V$ and $W$.  Recall that $L_1$ and $L_2$ were the one dimensional real representation of $C_4$ and $R$ was the two dimensional irreducible.  We let $V = L_1^{2^n} \oplus L_2^{2^n-\epsilon} \oplus R^l$ and $W = L_1^{2^n-1} \oplus L_2^{2^n-\epsilon} \oplus R^l$.

\propref{fixed point formula} tells us that
$$ \Gr_d(V)^{C_4} = \bigsqcup_{j+k+2i=d} \Gr_j(\R^{2^n}) \times \Gr_k(\R^{2^n-\epsilon})\times \Gr_i(\C^l)$$
and
$$ \Gr_d(W)^{C_4} = \bigsqcup_{j+k+2i=d} \Gr_j(\R^{2^n-1}) \times \Gr_k(\R^{2^n-\epsilon})\times \Gr_i(\C^l),$$
so that
$$ C_d(V,W)^{C_4} = \bigvee_{j+k+2i=d} C_j(\R^{2^n}) \sm \Gr_k(\R^{2^n-\epsilon})_+ \sm \Gr_i(\C^l)_+.$$

Thus we have
\begin{equation*}
\begin{split}
\bar k_n(C_d(\R^m)) &
\geq  \sum_{j+k+2i=d} \bar k_{n-1}(C_j(\R^{2^n}))k_{n-1}(\Gr_k(\R^{2^n-\epsilon}))k_{n-1}(\Gr_i(\C^{l})) \\
  & = \sum_{j+k+2i=d} \binom{2^n-1}{j-1}\binom{2^n-\epsilon}{k}\binom{l}{i} \text{ \ (using \thmref{C_d collapse thm} and \thmref{collapse thm})} \\
  & = \sum_i \left [ \sum_{j+k=d-2i} \binom{2^n-1}{j-1}\binom{2^n-\epsilon}{k} \right ] \binom{l}{i} \\
  & = \sum_i \binom{2^{n+1}-1-\epsilon}{d-1-2i}\binom{l}{i}.
\end{split}
\end{equation*}

\end{proof}

\begin{conj} \label{kn of cofib conj}Equality holds in \thmref{Cd lower bound thm}.
\end{conj}

As before, this would be implied by a conjectural calculation of the $Q_n$ homology of $C_d(\R_n)$.

\begin{conj} \label{Cofib collapsing conj} Let $m = 2^{n+1} - \epsilon + 2l$ with $\epsilon =$ 0 or 1, and $l \geq 0$.  Then
$$ \bar k_{Q_n}(C_d(\R^m)) = \sum_i \binom{2^{n+1}-1-\epsilon}{d-1-2i}\binom{l}{i}.$$.
\end{conj}

Our various conjectures imply a conjecture about the behavior of the boundary map
$$\delta: H^*(\Gr_d(\R^{m-1});Q_n) \ra  \widetilde H^{*+2^{n+1}-1}(C_d(\R^{m});Q_n),$$
when $m=2^{n+1}-\epsilon +2l$.  Let $k^{\delta}_n(d,m)$ denote the dimension of the image of this map.

\conjref{collapsing conj} says that $k_{Q_n}(\Gr_d(\R^m)) = k^G_n(d,m)$, where
$$k^G_n(d,m) = \sum_i \binom{2^{n+1}-\epsilon}{d-2i}\binom{l}{i}.$$

\conjref{Cofib collapsing conj} similarly says that $\bar k_{Q_n}(C_d(\R^m)) = \bar k^C_n(d,m)$, where $$\bar k^C_n(d,m) = \sum_i \binom{2^{n+1}-1-\epsilon}{d-1-2i}\binom{l}{i}.$$

If these conjectures are true, then the exactness of the $Q_n$--homology long exact sequence would imply that
$$ k^G_n(d,m) + 2 k^{\delta}_n(d,m) = k^G_n(d,m-1) + \bar k^C_n(d,m),$$
so that
$$  k^{\delta}_n(d,m) = \frac{1}{2} \left [k^G_n(d,m-1)  + \bar k^C_n(d,m) - k^G_n(d,m)\right ].$$

As expected, the right hand side here is zero if $m$ is even, i.e. $\epsilon = 0$.

When $m$ is odd, so $\epsilon = 1$, the right hand side is not zero, but can be rearranged as in the following lemma.

\begin{lem} \label{boundary lemma}  If $m=2^{n+1}-1+2l$ and $l>0$, then
$$ \frac{1}{2} \left [k^G_n(d,m-1)  + \bar k^C_n(d,m) - k^G_n(d,m)\right ] = \sum_i \binom{2^{n+1}-2}{d-1-2i}\binom{l-1}{i}.$$
\end{lem}
\begin{proof}  We expand $k^G_n(d,m-1)$:
\begin{equation*}
\begin{split}
k^G_n(d,m-1) &
 = \sum_i \binom{2^{n+1}}{d-2i}\binom{l-1}{i} \\
  & = \sum_i \left [\binom{2^{n+1}-2}{d-2i}+ 2\binom{2^{n+1}-2}{d-1-2i}+\binom{2^{n+1}-2}{d-2-2i}\right ]\binom{l-1}{i}.
\end{split}
\end{equation*}

We rewrite $k^G_n(d,m) - \bar k^C_n(d,m)$:
\begin{equation*}
\begin{split}
k^G_n(d,m) - \bar k^C_n(d,m) &
 = \sum_i \left [\binom{2^{n+1}-1}{d-2i}-\binom{2^{n+1}-2}{d-1-2i}\right ]\binom{l}{i} \\
 & = \sum_i \binom{2^{n+1}-2}{d-2i}\binom{l}{i} \\
 & = \sum_i \binom{2^{n+1}-2}{d-2i}\left [\binom{l-1}{i}+ \binom{l-1}{i-1}\right ]\\
 & = \sum_i \left [\binom{2^{n+1}-2}{d-2i}+\binom{2^{n+1}-2}{d-2-2i}\right ]\binom{l-1}{i}.
\end{split}
\end{equation*}

Subtracting our second expression from the first, and dividing by two, proves the lemma.
\end{proof}

Thus we can add the following to our conjectures.
\begin{conj} \label{boundary conj} If $m=2^{n+1}-1+2l$ and $l>0$, then
$$ k^{\delta}_n(d,m) = \sum_i \binom{2^{n+1}-2}{d-1-2i}\binom{l-1}{i}.$$
\end{conj}

\begin{ex}  Suppose that $n=0$, so $m=2l+1$.  \conjref{Cofib collapsing conj} predicts that
\begin{equation*}
\dim_{\Q} H^*(C_d(\R^{2l+1});\Q) =
\begin{cases}
0 & \text{if } d \text{ is even} \\ \binom{l}{c} & \text{if } d=2c+1.
\end{cases}
\end{equation*}
Similarly, \conjref{boundary conj} predicts that
\begin{equation*}
k^{\delta}_0(d,2l+1) =
\begin{cases}
0 & \text{if } d \text{ is even} \\ \binom{l-1}{c} & \text{if } d=2c+1.
\end{cases}
\end{equation*}

Noting that $k^{\delta}_0(d,2l+1)$ can be viewed as the dimension of the cokernel of the map
$$ i^*: H^*(\Gr_{d}(\R^{2l+1});\Q) \ra H^*(\Gr_{d}(\R^{2l});\Q),$$
one can check that our conjectures do correspond to the known behavior of $i^*$ -- it takes Pontryagin classes to Pontryagin classes -- together with the computations
\begin{equation*}
\dim_{\Q} H^*(\Gr_{d}(\R^{m});\Q) =
\begin{cases}
\binom{l}{c} & \text{if } m=2l+1 \text{ and } d=2c \text{ or } 2c+1 \\ 2 \binom{l-1}{c} & \text{if }  m=2l \text{ and } d=2c+1.
\end{cases}
\end{equation*}
\end{ex}

\appendix
\newpage
\section{Tables}
In this appendix we present some tables of calculations that support
Conjecture~\ref{collapsing conj}.   For larger Grassmannians the main
obstacle to testing the conjecture is the sheer volume of calculations
needed. To expedite these calculations the authors used the University
of Virginia Rivanna High Performance computing system. The white cells are the
conjectured values which have not been checked due to computational
limitations. The tables are necessarily symmetric in \(c\) and \(d\).

\newcommand{\tablekey}[2]{\fcolorbox{black}{#1}{\phantom{x}} #2}
\newcommand{\tablelegend}{{\small \medskip
\begin{tabular}{l l l l}
 \tablekey{cyan}{\(cd \le 2^{n+1}-1\)}  &
\tablekey{gray}{Projective Spaces} &
\tablekey{yellow}{\thmref{d=2 thm}} \\
\tablekey{light-gray}{Conjecture Verified}  &
\tablekey{white}{Conjecture} &
\tablekey{green}{\thmref{collapse thm}}
\end{tabular}
}}
\definecolor{light-gray}{gray}{0.80}
{\begin{center} \(k_{1}(\Gr_d(\R^{d+c}))\)\\ \medskip {\tiny
\begin{tabular}{|c|ccccccccccc|}\hline
\diagbox{$d$}{$c$} & 1 & 2 & 3 & 4 & 5 & 6 & 7 & 8 & 9 & 10 & 11\\ \hline
1 & \cellcolor{cyan}2 & \cellcolor{cyan}3 & \cellcolor{cyan}4 & \cellcolor{gray}3 & \cellcolor{gray}4 & \cellcolor{gray}3 & \cellcolor{gray}4 & \cellcolor{gray}3 & \cellcolor{gray}4 & \cellcolor{gray}3 & \cellcolor{gray}4\\
2 & \cellcolor{cyan}3 & \cellcolor{yellow}6 & \cellcolor{yellow}4 & \cellcolor{yellow}7 & \cellcolor{yellow}5 & \cellcolor{yellow}8 & \cellcolor{yellow}6 & \cellcolor{yellow}9 & \cellcolor{yellow}7 & \cellcolor{yellow}10 & \cellcolor{yellow}8\\
3 & \cellcolor{cyan}4 & \cellcolor{yellow}4 & \cellcolor{light-gray}8 & \cellcolor{light-gray}7 & \cellcolor{light-gray}12 & \cellcolor{light-gray}10 & \cellcolor{light-gray}16 & \cellcolor{light-gray}13 & \cellcolor{light-gray}20 & \cellcolor{light-gray}16 & \cellcolor{light-gray}24\\
4 & \cellcolor{gray}3 & \cellcolor{yellow}7 & \cellcolor{light-gray}7 & \cellcolor{light-gray}14 & \cellcolor{light-gray}12 & \cellcolor{light-gray}22 & \cellcolor{light-gray}18 & \cellcolor{light-gray}31 & \cellcolor{light-gray}25 & \cellcolor{light-gray}41 & \cellcolor{light-gray}33\\
5 & \cellcolor{gray}4 & \cellcolor{yellow}5 & \cellcolor{light-gray}12 & \cellcolor{light-gray}12 & \cellcolor{light-gray}24 & \cellcolor{light-gray}22 & \cellcolor{light-gray}40 & \cellcolor{light-gray}35 & \cellcolor{light-gray}60 & \cellcolor{light-gray}51 & \cellcolor{light-gray}84\\
6 & \cellcolor{gray}3 & \cellcolor{yellow}8 & \cellcolor{light-gray}10 & \cellcolor{light-gray}22 & \cellcolor{light-gray}22 & \cellcolor{light-gray}44 & \cellcolor{light-gray}40 & \cellcolor{light-gray}75 & \cellcolor{light-gray}65 & \cellcolor{light-gray}116 & \cellcolor{light-gray}98\\
7 & \cellcolor{gray}4 & \cellcolor{yellow}6 & \cellcolor{light-gray}16 & \cellcolor{light-gray}18 & \cellcolor{light-gray}40 & \cellcolor{light-gray}40 & \cellcolor{light-gray}80 & \cellcolor{light-gray}75 & \cellcolor{light-gray}140 & \cellcolor{light-gray}126 & \cellcolor{light-gray}224\\
8 & \cellcolor{gray}3 & \cellcolor{yellow}9 & \cellcolor{light-gray}13 & \cellcolor{light-gray}31 & \cellcolor{light-gray}35 & \cellcolor{light-gray}75 & \cellcolor{light-gray}75 & \cellcolor{light-gray}150 & \cellcolor{light-gray}140 & \cellcolor{light-gray}266 & \cellcolor{light-gray}238\\
9 & \cellcolor{gray}4 & \cellcolor{yellow}7 & \cellcolor{light-gray}20 & \cellcolor{light-gray}25 & \cellcolor{light-gray}60 & \cellcolor{light-gray}65 & \cellcolor{light-gray}140 & \cellcolor{light-gray}140 & \cellcolor{light-gray}280 & \cellcolor{light-gray}266 & \cellcolor{light-gray}504\\
10 & \cellcolor{gray}3 & \cellcolor{yellow}10 & \cellcolor{light-gray}16 & \cellcolor{light-gray}41 & \cellcolor{light-gray}51 & \cellcolor{light-gray}116 & \cellcolor{light-gray}126 & \cellcolor{light-gray}266 & \cellcolor{light-gray}266 & \cellcolor{light-gray}532 & \cellcolor{light-gray}504\\
11 & \cellcolor{gray}4 & \cellcolor{yellow}8 & \cellcolor{light-gray}24 & \cellcolor{light-gray}33 & \cellcolor{light-gray}84 & \cellcolor{light-gray}98 & \cellcolor{light-gray}224 & \cellcolor{light-gray}238 & \cellcolor{light-gray}504 & \cellcolor{light-gray}504 & 1008\\
12 & \cellcolor{gray}3 & \cellcolor{yellow}11 & \cellcolor{light-gray}19 & \cellcolor{light-gray}52 & \cellcolor{light-gray}70 & \cellcolor{light-gray}168 & \cellcolor{light-gray}196 & \cellcolor{light-gray}434 & \cellcolor{light-gray}462 & 966 & 966\\
13 & \cellcolor{gray}4 & \cellcolor{yellow}9 & \cellcolor{light-gray}28 & \cellcolor{light-gray}42 & \cellcolor{light-gray}112 & \cellcolor{light-gray}140 & \cellcolor{light-gray}336 & \cellcolor{light-gray}378 & \cellcolor{light-gray}840 & 882 & 1848\\
14 & \cellcolor{gray}3 & \cellcolor{yellow}12 & \cellcolor{light-gray}22 & \cellcolor{light-gray}64 & \cellcolor{light-gray}92 & \cellcolor{light-gray}232 & \cellcolor{light-gray}288 & \cellcolor{light-gray}666 & \cellcolor{light-gray}750 & 1632 & 1716\\
15 & \cellcolor{gray}4 & \cellcolor{yellow}10 & \cellcolor{light-gray}32 & \cellcolor{light-gray}52 & \cellcolor{light-gray}144 & \cellcolor{light-gray}192 & 480 & 570 & 1320 & 1452 & 3168\\
16 & \cellcolor{gray}3 & \cellcolor{yellow}13 & \cellcolor{light-gray}25 & \cellcolor{light-gray}77 & \cellcolor{light-gray}117 & \cellcolor{light-gray}309 & 405 & 975 & 1155 & 2607 & 2871\\
17 & \cellcolor{gray}4 & \cellcolor{yellow}11 & \cellcolor{light-gray}36 & \cellcolor{light-gray}63 & \cellcolor{light-gray}180 & \cellcolor{light-gray}255 & 660 & 825 & 1980 & 2277 & 5148\\
18 & \cellcolor{gray}3 & \cellcolor{yellow}14 & \cellcolor{light-gray}28 & \cellcolor{light-gray}91 & \cellcolor{light-gray}145 & \cellcolor{light-gray}400 & 550 & 1375 & 1705 & 3982 & 4576\\
19 & \cellcolor{gray}4 & \cellcolor{yellow}12 & \cellcolor{light-gray}40 & \cellcolor{light-gray}75 & \cellcolor{light-gray}220 & \cellcolor{light-gray}330 & 880 & 1155 & 2860 & 3432 & 8008\\
20 & \cellcolor{gray}3 & \cellcolor{yellow}15 & \cellcolor{light-gray}31 & \cellcolor{light-gray}106 & \cellcolor{light-gray}176 & \cellcolor{light-gray}506 & 726 & 1881 & 2431 & 5863 & 7007\\
21 & \cellcolor{gray}4 & \cellcolor{yellow}13 & \cellcolor{light-gray}44 & \cellcolor{light-gray}88 & \cellcolor{light-gray}264 & \cellcolor{light-gray}418 & 1144 & 1573 & 4004 & 5005 & 12012\\
22 & \cellcolor{gray}3 & \cellcolor{yellow}16 & \cellcolor{light-gray}34 & \cellcolor{light-gray}122 & \cellcolor{light-gray}210 & 628 & 936 & 2509 & 3367 & 8372 & 10374\\
23 & \cellcolor{gray}4 & \cellcolor{yellow}14 & \cellcolor{light-gray}48 & \cellcolor{light-gray}102 & \cellcolor{light-gray}312 & 520 & 1456 & 2093 & 5460 & 7098 & 17472\\
24 & \cellcolor{gray}3 & \cellcolor{yellow}17 & \cellcolor{light-gray}37 & \cellcolor{light-gray}139 & \cellcolor{light-gray}247 & 767 & 1183 & 3276 & 4550 & 11648 & 14924\\
25 & \cellcolor{gray}4 & \cellcolor{yellow}15 & \cellcolor{light-gray}52 & \cellcolor{light-gray}117 & \cellcolor{light-gray}364 & 637 & 1820 & 2730 & 7280 & 9828 & 24752\\
26 & \cellcolor{gray}3 & \cellcolor{yellow}18 & \cellcolor{light-gray}40 & \cellcolor{light-gray}157 & \cellcolor{light-gray}287 & 924 & 1470 & 4200 & 6020 & 15848 & 20944\\
27 & \cellcolor{gray}4 & \cellcolor{yellow}16 & \cellcolor{light-gray}56 & \cellcolor{light-gray}133 & \cellcolor{light-gray}420 & 770 & 2240 & 3500 & 9520 & 13328 & 34272\\
28 & \cellcolor{gray}3 & \cellcolor{yellow}19 & \cellcolor{light-gray}43 & \cellcolor{light-gray}176 & \cellcolor{light-gray}330 & 1100 & 1800 & 5300 & 7820 & 21148 & 28764\\
29 & \cellcolor{gray}4 & \cellcolor{yellow}17 & \cellcolor{light-gray}60 & \cellcolor{light-gray}150 & \cellcolor{light-gray}480 & 920 & 2720 & 4420 & 12240 & 17748 & 46512\\
30 & \cellcolor{gray}3 & \cellcolor{yellow}20 & \cellcolor{light-gray}46 & \cellcolor{light-gray}196 & \cellcolor{light-gray}376 & 1296 & 2176 & 6596 & 9996 & 27744 & 38760\\
31 & \cellcolor{gray}4 & \cellcolor{yellow}18 & \cellcolor{light-gray}64 & \cellcolor{light-gray}168 & \cellcolor{light-gray}544 & 1088 & 3264 & 5508 & 15504 & 23256 & 62016\\
32 & \cellcolor{gray}3 & \cellcolor{yellow}21 & \cellcolor{light-gray}49 & \cellcolor{light-gray}217 & \cellcolor{light-gray}425 & 1513 & 2601 & 8109 & 12597 & 35853 & 51357\\
33 & \cellcolor{gray}4 & \cellcolor{yellow}19 & \cellcolor{light-gray}68 & \cellcolor{light-gray}187 & \cellcolor{light-gray}612 & 1275 & 3876 & 6783 & 19380 & 30039 & 81396\\
34 & \cellcolor{gray}3 & \cellcolor{yellow}22 & \cellcolor{light-gray}52 & \cellcolor{light-gray}239 & 477 & 1752 & 3078 & 9861 & 15675 & 45714 & 67032\\
35 & \cellcolor{gray}4 & \cellcolor{yellow}20 & \cellcolor{light-gray}72 & \cellcolor{light-gray}207 & 684 & 1482 & 4560 & 8265 & 23940 & 38304 & 105336\\
36 & \cellcolor{gray}3 & \cellcolor{yellow}23 & \cellcolor{light-gray}55 & \cellcolor{light-gray}262 & 532 & 2014 & 3610 & 11875 & 19285 & 57589 & 86317\\
37 & \cellcolor{gray}4 & \cellcolor{yellow}21 & \cellcolor{light-gray}76 & \cellcolor{light-gray}228 & 760 & 1710 & 5320 & 9975 & 29260 & 48279 & 134596\\
38 & \cellcolor{gray}3 & \cellcolor{yellow}24 & \cellcolor{light-gray}58 & \cellcolor{light-gray}286 & 590 & 2300 & 4200 & 14175 & 23485 & 71764 & 109802\\
39 & \cellcolor{gray}4 & \cellcolor{yellow}22 & \cellcolor{light-gray}80 & \cellcolor{light-gray}250 & 840 & 1960 & 6160 & 11935 & 35420 & 60214 & 170016\\
40 & \cellcolor{gray}3 & \cellcolor{yellow}25 & \cellcolor{light-gray}61 & \cellcolor{light-gray}311 & 651 & 2611 & 4851 & 16786 & 28336 & 88550 & 138138\\
41 & \cellcolor{gray}4 & \cellcolor{yellow}23 & \cellcolor{light-gray}84 & \cellcolor{light-gray}273 & 924 & 2233 & 7084 & 14168 & 42504 & 74382 & 212520\\
42 & \cellcolor{gray}3 & \cellcolor{yellow}26 & \cellcolor{light-gray}64 & \cellcolor{light-gray}337 & 715 & 2948 & 5566 & 19734 & 33902 & 108284 & 172040\\
43 & \cellcolor{gray}4 & \cellcolor{yellow}24 & \cellcolor{light-gray}88 & \cellcolor{light-gray}297 & 1012 & 2530 & 8096 & 16698 & 50600 & 91080 & 263120\\
44 & \cellcolor{gray}3 & \cellcolor{yellow}27 & \cellcolor{light-gray}67 & \cellcolor{light-gray}364 & 782 & 3312 & 6348 & 23046 & 40250 & 131330 & 212290\\
45 & \cellcolor{gray}4 & \cellcolor{yellow}25 & \cellcolor{light-gray}92 & \cellcolor{light-gray}322 & 1104 & 2852 & 9200 & 19550 & 59800 & 110630 & 322920\\
46 & \cellcolor{gray}3 & \cellcolor{yellow}28 & \cellcolor{light-gray}70 & \cellcolor{light-gray}392 & 852 & 3704 & 7200 & 26750 & 47450 & 158080 & 259740\\
47 & \cellcolor{gray}4 & \cellcolor{yellow}26 & \cellcolor{light-gray}96 & \cellcolor{light-gray}348 & 1200 & 3200 & 10400 & 22750 & 70200 & 133380 & 393120\\
48 & \cellcolor{gray}3 & \cellcolor{yellow}29 & \cellcolor{light-gray}73 & \cellcolor{light-gray}421 & 925 & 4125 & 8125 & 30875 & 55575 & 188955 & 315315\\
49 & \cellcolor{gray}4 & \cellcolor{yellow}27 & \cellcolor{light-gray}100 & \cellcolor{light-gray}375 & 1300 & 3575 & 11700 & 26325 & 81900 & 159705 & 475020\\
50 & \cellcolor{gray}3 & \cellcolor{yellow}30 & \cellcolor{light-gray}76 & \cellcolor{light-gray}451 & 1001 & 4576 & 9126 & 35451 & 64701 & 224406 & 380016\\
\hline \end{tabular}}\end{center}

\medskip
\begin{center}
\begin{tabular}{l l l l}
 \tablekey{cyan}{\(cd \le 2^{n+1}-1\)}  &
\tablekey{gray}{Projective Spaces} &
\tablekey{yellow}{\thmref{d=2 thm}} \\
\tablekey{light-gray}{Conjecture Verified}  &
\tablekey{white}{Conjecture}
\end{tabular}
\end{center} \newpage
\definecolor{light-gray}{gray}{0.80}

{\begin{center} \(k_{2}(\Gr_d(\R^{d+c}))\)\\ \medskip {\tiny
\begin{tabular}{|c|ccccccccccc|}\hline
\diagbox{$d$}{$c$} & 1 & 2 & 3 & 4 & 5 & 6 & 7 & 8 & 9 & 10 & 11\\ \hline
1 & \cellcolor{cyan}2 & \cellcolor{cyan}3 & \cellcolor{cyan}4 & \cellcolor{cyan}5 & \cellcolor{cyan}6 & \cellcolor{cyan}7 & \cellcolor{cyan}8 & \cellcolor{gray}7 & \cellcolor{gray}8 & \cellcolor{gray}7 & \cellcolor{gray}8\\
2 & \cellcolor{cyan}3 & \cellcolor{cyan}6 & \cellcolor{cyan}10 & \cellcolor{green}15 & \cellcolor{green}21 & \cellcolor{yellow}28 & \cellcolor{yellow}22 & \cellcolor{yellow}29 & \cellcolor{yellow}23 & \cellcolor{yellow}30 & \cellcolor{yellow}24\\
3 & \cellcolor{cyan}4 & \cellcolor{cyan}10 & \cellcolor{green}20 & \cellcolor{green}35 & \cellcolor{light-gray}56 & \cellcolor{light-gray}42 & \cellcolor{light-gray}64 & \cellcolor{light-gray}49 & \cellcolor{light-gray}72 & \cellcolor{light-gray}56 & \cellcolor{light-gray}80\\
4 & \cellcolor{cyan}5 & \cellcolor{green}15 & \cellcolor{green}35 & \cellcolor{light-gray}70 & \cellcolor{light-gray}56 & \cellcolor{light-gray}98 & \cellcolor{light-gray}78 & \cellcolor{light-gray}127 & \cellcolor{light-gray}101 & \cellcolor{light-gray}157 & \cellcolor{light-gray}125\\
5 & \cellcolor{cyan}6 & \cellcolor{green}21 & \cellcolor{light-gray}56 & \cellcolor{light-gray}56 & \cellcolor{light-gray}112 & \cellcolor{light-gray}98 & \cellcolor{light-gray}176 & \cellcolor{light-gray}147 & \cellcolor{light-gray}248 & \cellcolor{light-gray}203 & \cellcolor{light-gray}328\\
6 & \cellcolor{cyan}7 & \cellcolor{yellow}28 & \cellcolor{light-gray}42 & \cellcolor{light-gray}98 & \cellcolor{light-gray}98 & \cellcolor{light-gray}196 & \cellcolor{light-gray}176 & \cellcolor{light-gray}323 & \cellcolor{light-gray}277 & \cellcolor{light-gray}480 & \cellcolor{light-gray}402\\
7 & \cellcolor{cyan}8 & \cellcolor{yellow}22 & \cellcolor{light-gray}64 & \cellcolor{light-gray}78 & \cellcolor{light-gray}176 & \cellcolor{light-gray}176 & 352 & 323 & 600 & 526 & 928\\
8 & \cellcolor{gray}7 & \cellcolor{yellow}29 & \cellcolor{light-gray}49 & \cellcolor{light-gray}127 & \cellcolor{light-gray}147 & \cellcolor{light-gray}323 & 323 & 646 & 600 & 1126 & 1002\\
9 & \cellcolor{gray}8 & \cellcolor{yellow}23 & \cellcolor{light-gray}72 & \cellcolor{light-gray}101 & \cellcolor{light-gray}248 & \cellcolor{light-gray}277 & 600 & 600 & 1200 & 1126 & 2128\\
10 & \cellcolor{gray}7 & \cellcolor{yellow}30 & \cellcolor{light-gray}56 & \cellcolor{light-gray}157 & \cellcolor{light-gray}203 & \cellcolor{light-gray}480 & 526 & 1126 & 1126 & 2252 & 2128\\
11 & \cellcolor{gray}8 & \cellcolor{yellow}24 & \cellcolor{light-gray}80 & \cellcolor{light-gray}125 & \cellcolor{light-gray}328 & \cellcolor{light-gray}402 & 928 & 1002 & 2128 & 2128 & 4256\\
12 & \cellcolor{gray}7 & \cellcolor{yellow}31 & \cellcolor{light-gray}63 & \cellcolor{light-gray}188 & \cellcolor{light-gray}266 & \cellcolor{light-gray}668 & 792 & 1794 & 1918 & 4046 & 4046\\
13 & \cellcolor{gray}8 & \cellcolor{yellow}25 & \cellcolor{light-gray}88 & \cellcolor{light-gray}150 & \cellcolor{light-gray}416 & \cellcolor{light-gray}552 & 1344 & 1554 & 3472 & 3682 & 7728\\
14 & \cellcolor{gray}7 & \cellcolor{yellow}32 & \cellcolor{light-gray}70 & \cellcolor{light-gray}220 & \cellcolor{light-gray}336 & \cellcolor{light-gray}888 & 1128 & 2682 & 3046 & 6728 & 7092\\
15 & \cellcolor{gray}8 & \cellcolor{yellow}26 & \cellcolor{light-gray}96 & \cellcolor{light-gray}176 & \cellcolor{light-gray}512 & \cellcolor{light-gray}728 & 1856 & 2282 & 5328 & 5964 & 13056\\
16 & \cellcolor{gray}7 & \cellcolor{yellow}33 & \cellcolor{light-gray}77 & \cellcolor{light-gray}253 & \cellcolor{light-gray}413 & \cellcolor{light-gray}1141 & 1541 & 3823 & 4587 & 10551 & 11679\\
17 & \cellcolor{gray}8 & \cellcolor{yellow}27 & \cellcolor{light-gray}104 & \cellcolor{light-gray}203 & \cellcolor{light-gray}616 & \cellcolor{light-gray}931 & 2472 & 3213 & 7800 & 9177 & 20856\\
18 & \cellcolor{gray}7 & \cellcolor{yellow}34 & \cellcolor{light-gray}84 & \cellcolor{light-gray}287 & \cellcolor{light-gray}497 & \cellcolor{light-gray}1428 & 2038 & 5251 & 6625 & 15802 & 18304\\
19 & \cellcolor{gray}8 & \cellcolor{yellow}28 & \cellcolor{light-gray}112 & \cellcolor{light-gray}231 & \cellcolor{light-gray}728 & \cellcolor{light-gray}1162 & 3200 & 4375 & 11000 & 13552 & 31856\\
20 & \cellcolor{gray}7 & \cellcolor{yellow}35 & \cellcolor{light-gray}91 & \cellcolor{light-gray}322 & \cellcolor{light-gray}588 & \cellcolor{light-gray}1750 & 2626 & 7001 & 9251 & 22803 & 27555\\
21 & \cellcolor{gray}8 & \cellcolor{yellow}29 & \cellcolor{light-gray}120 & \cellcolor{light-gray}260 & \cellcolor{light-gray}848 & \cellcolor{light-gray}1422 & 4048 & 5797 & 15048 & 19349 & 46904\\
22 & \cellcolor{gray}7 & \cellcolor{yellow}36 & \cellcolor{light-gray}98 & \cellcolor{light-gray}358 & \cellcolor{light-gray}686 & 2108 & 3312 & 9109 & 12563 & 31912 & 40118\\
23 & \cellcolor{gray}8 & \cellcolor{yellow}30 & \cellcolor{light-gray}128 & \cellcolor{light-gray}290 & \cellcolor{light-gray}976 & 1712 & 5024 & 7509 & 20072 & 26858 & 66976\\
24 & \cellcolor{gray}7 & \cellcolor{yellow}37 & \cellcolor{light-gray}105 & \cellcolor{light-gray}395 & \cellcolor{light-gray}791 & 2503 & 4103 & 11612 & 16666 & 43524 & 56784\\
25 & \cellcolor{gray}8 & \cellcolor{yellow}31 & \cellcolor{light-gray}136 & \cellcolor{light-gray}321 & \cellcolor{light-gray}1112 & 2033 & 6136 & 9542 & 26208 & 36400 & 93184\\
26 & \cellcolor{gray}7 & \cellcolor{yellow}38 & \cellcolor{light-gray}112 & \cellcolor{light-gray}433 & \cellcolor{light-gray}903 & 2936 & 5006 & 14548 & 21672 & 58072 & 78456\\
27 & \cellcolor{gray}8 & \cellcolor{yellow}32 & \cellcolor{light-gray}144 & \cellcolor{light-gray}353 & \cellcolor{light-gray}1256 & 2386 & 7392 & 11928 & 33600 & 48328 & 126784\\
28 & \cellcolor{gray}7 & \cellcolor{yellow}39 & \cellcolor{light-gray}119 & \cellcolor{light-gray}472 & \cellcolor{light-gray}1022 & 3408 & 6028 & 17956 & 27700 & 76028 & 106156\\
29 & \cellcolor{gray}8 & \cellcolor{yellow}33 & \cellcolor{light-gray}152 & \cellcolor{light-gray}386 & \cellcolor{light-gray}1408 & 2772 & 8800 & 14700 & 42400 & 63028 & 169184\\
30 & \cellcolor{gray}7 & \cellcolor{yellow}40 & \cellcolor{light-gray}126 & \cellcolor{light-gray}512 & \cellcolor{light-gray}1148 & 3920 & 7176 & 21876 & 34876 & 97904 & 141032\\
31 & \cellcolor{gray}8 & \cellcolor{yellow}34 & \cellcolor{light-gray}160 & \cellcolor{light-gray}420 & \cellcolor{light-gray}1568 & 3192 & 10368 & 17892 & 52768 & 80920 & 221952\\
32 & \cellcolor{gray}7 & \cellcolor{yellow}41 & \cellcolor{light-gray}133 & \cellcolor{light-gray}553 & \cellcolor{light-gray}1281 & 4473 & 8457 & 26349 & 43333 & 124253 & 184365\\
33 & \cellcolor{gray}8 & \cellcolor{yellow}35 & \cellcolor{light-gray}168 & \cellcolor{light-gray}455 & \cellcolor{light-gray}1736 & 3647 & 12104 & 21539 & 64872 & 102459 & 286824\\
34 & \cellcolor{gray}7 & \cellcolor{yellow}42 & \cellcolor{light-gray}140 & \cellcolor{light-gray}595 & \cellcolor{light-gray}1421 & 5068 & 9878 & 31417 & 53211 & 155670 & 237576\\
35 & \cellcolor{gray}8 & \cellcolor{yellow}36 & \cellcolor{light-gray}176 & \cellcolor{light-gray}491 & \cellcolor{light-gray}1912 & 4138 & 14016 & 25677 & 78888 & 128136 & 365712\\
36 & \cellcolor{gray}7 & \cellcolor{yellow}43 & \cellcolor{light-gray}147 & \cellcolor{light-gray}638 & \cellcolor{light-gray}1568 & 5706 & 11446 & 37123 & 64657 & 192793 & 302233\\
37 & \cellcolor{gray}8 & \cellcolor{yellow}37 & \cellcolor{light-gray}184 & \cellcolor{light-gray}528 & \cellcolor{light-gray}2096 & 4666 & 16112 & 30343 & 95000 & 158479 & 460712\\
38 & \cellcolor{gray}7 & \cellcolor{yellow}44 & \cellcolor{light-gray}154 & \cellcolor{light-gray}682 & 1722 & 6388 & 13168 & 43511 & 77825 & 236304 & 380058\\
39 & \cellcolor{gray}8 & \cellcolor{yellow}38 & \cellcolor{light-gray}192 & \cellcolor{light-gray}566 & 2288 & 5232 & 18400 & 35575 & 113400 & 194054 & 574112\\
40 & \cellcolor{gray}7 & \cellcolor{yellow}45 & \cellcolor{light-gray}161 & \cellcolor{light-gray}727 & 1883 & 7115 & 15051 & 50626 & 92876 & 286930 & 472934\\
41 & \cellcolor{gray}8 & \cellcolor{yellow}39 & \cellcolor{light-gray}200 & \cellcolor{light-gray}605 & 2488 & 5837 & 20888 & 41412 & 134288 & 235466 & 708400\\
42 & \cellcolor{gray}7 & \cellcolor{yellow}46 & \cellcolor{light-gray}168 & \cellcolor{light-gray}773 & 2051 & 7888 & 17102 & 58514 & 109978 & 345444 & 582912\\
43 & \cellcolor{gray}8 & \cellcolor{yellow}40 & \cellcolor{light-gray}208 & \cellcolor{light-gray}645 & 2696 & 6482 & 23584 & 47894 & 157872 & 283360 & 866272\\
44 & \cellcolor{gray}7 & \cellcolor{yellow}47 & \cellcolor{light-gray}175 & \cellcolor{light-gray}820 & 2226 & 8708 & 19328 & 67222 & 129306 & 412666 & 712218\\
45 & \cellcolor{gray}8 & \cellcolor{yellow}41 & \cellcolor{light-gray}216 & \cellcolor{light-gray}686 & 2912 & 7168 & 26496 & 55062 & 184368 & 338422 & 1050640\\
46 & \cellcolor{gray}7 & \cellcolor{yellow}48 & \cellcolor{light-gray}182 & \cellcolor{light-gray}868 & 2408 & 9576 & 21736 & 76798 & 151042 & 489464 & 863260\\
47 & \cellcolor{gray}8 & \cellcolor{yellow}42 & \cellcolor{light-gray}224 & \cellcolor{light-gray}728 & 3136 & 7896 & 29632 & 62958 & 214000 & 401380 & 1264640\\
48 & \cellcolor{gray}7 & \cellcolor{yellow}49 & \cellcolor{light-gray}189 & \cellcolor{light-gray}917 & 2597 & 10493 & 24333 & 87291 & 175375 & 576755 & 1038635\\
49 & \cellcolor{gray}8 & \cellcolor{yellow}43 & \cellcolor{light-gray}232 & \cellcolor{light-gray}771 & 3368 & 8667 & 33000 & 71625 & 247000 & 473005 & 1511640\\
50 & \cellcolor{gray}7 & \cellcolor{yellow}50 & \cellcolor{light-gray}196 & \cellcolor{light-gray}967 & 2793 & 11460 & 27126 & 98751 & 202501 & 675506 & 1241136\\
\hline \end{tabular}}

\tablelegend
\end{center}

 \newpage
\definecolor{light-gray}{gray}{0.80}

{\begin{center} \(k_{3}(\Gr_d(\R^{d+c}))\)\\ \medskip {\tiny
\begin{tabular}{|c|ccccccccccc|}\hline
\diagbox{$d$}{$c$} & 1 & 2 & 3 & 4 & 5 & 6 & 7 & 8 & 9 & 10 & 11\\ \hline
1 & \cellcolor{cyan}2 & \cellcolor{cyan}3 & \cellcolor{cyan}4 & \cellcolor{cyan}5 & \cellcolor{cyan}6 & \cellcolor{cyan}7 & \cellcolor{cyan}8 & \cellcolor{cyan}9 & \cellcolor{cyan}10 & \cellcolor{cyan}11 & \cellcolor{cyan}12\\
2 & \cellcolor{cyan}3 & \cellcolor{cyan}6 & \cellcolor{cyan}10 & \cellcolor{cyan}15 & \cellcolor{cyan}21 & \cellcolor{cyan}28 & \cellcolor{cyan}36 & \cellcolor{green}45 & \cellcolor{green}55 & \cellcolor{green}66 & \cellcolor{green}78\\
3 & \cellcolor{cyan}4 & \cellcolor{cyan}10 & \cellcolor{cyan}20 & \cellcolor{cyan}35 & \cellcolor{cyan}56 & \cellcolor{green}84 & \cellcolor{green}120 & \cellcolor{green}165 & \cellcolor{green}220 & \cellcolor{green}286 & \cellcolor{green}364\\
4 & \cellcolor{cyan}5 & \cellcolor{cyan}15 & \cellcolor{cyan}35 & \cellcolor{green}70 & \cellcolor{green}126 & \cellcolor{green}210 & \cellcolor{green}330 & \cellcolor{green}495 & \cellcolor{green}715 & \cellcolor{green}1001 & \cellcolor{green}1365\\
5 & \cellcolor{cyan}6 & \cellcolor{cyan}21 & \cellcolor{cyan}56 & \cellcolor{green}126 & \cellcolor{green}252 & \cellcolor{green}462 & \cellcolor{green}792 & \cellcolor{green}1287 & \cellcolor{green}2002 & \cellcolor{green}3003 & \cellcolor{light-gray}4368\\
6 & \cellcolor{cyan}7 & \cellcolor{cyan}28 & \cellcolor{green}84 & \cellcolor{green}210 & \cellcolor{green}462 & \cellcolor{green}924 & \cellcolor{green}1716 & \cellcolor{green}3003 & \cellcolor{green}5005 & \cellcolor{light-gray}8008 & \cellcolor{light-gray}6370\\
7 & \cellcolor{cyan}8 & \cellcolor{cyan}36 & \cellcolor{green}120 & \cellcolor{green}330 & \cellcolor{green}792 & \cellcolor{green}1716 & \cellcolor{green}3432 & \cellcolor{green}6435 & \cellcolor{light-gray}11440 & \cellcolor{light-gray}9438 & \cellcolor{light-gray}15808\\
8 & \cellcolor{cyan}9 & \cellcolor{green}45 & \cellcolor{green}165 & \cellcolor{green}495 & \cellcolor{green}1287 & \cellcolor{green}3003 & \cellcolor{green}6435 & \cellcolor{light-gray}12870 & \cellcolor{light-gray}11440 & \cellcolor{light-gray}20878 & \cellcolor{light-gray}17810\\
9 & \cellcolor{cyan}10 & \cellcolor{green}55 & \cellcolor{green}220 & \cellcolor{green}715 & \cellcolor{green}2002 & \cellcolor{green}5005 & \cellcolor{light-gray}11440 & \cellcolor{light-gray}11440 & \cellcolor{light-gray}22880 & \cellcolor{light-gray}20878 & \cellcolor{light-gray}38688\\
10 & \cellcolor{cyan}11 & \cellcolor{green}66 & \cellcolor{green}286 & \cellcolor{green}1001 & \cellcolor{green}3003 & \cellcolor{light-gray}8008 & \cellcolor{light-gray}9438 & \cellcolor{light-gray}20878 & \cellcolor{light-gray}20878 & 41756 & 38688\\
11 & \cellcolor{cyan}12 & \cellcolor{green}78 & \cellcolor{green}364 & \cellcolor{green}1365 & \cellcolor{light-gray}4368 & \cellcolor{light-gray}6370 & \cellcolor{light-gray}15808 & \cellcolor{light-gray}17810 & \cellcolor{light-gray}38688 & 38688 & 77376\\
12 & \cellcolor{cyan}13 & \cellcolor{green}91 & \cellcolor{green}455 & \cellcolor{light-gray}1820 & \cellcolor{light-gray}3458 & \cellcolor{light-gray}9828 & \cellcolor{light-gray}12896 & \cellcolor{light-gray}30706 & 33774 & 72462 & 72462\\
13 & \cellcolor{cyan}14 & \cellcolor{green}105 & \cellcolor{light-gray}560 & \cellcolor{light-gray}1470 & \cellcolor{light-gray}4928 & \cellcolor{light-gray}7840 & \cellcolor{light-gray}20736 & \cellcolor{light-gray}25650 & 59424 & 64338 & 136800\\
14 & \cellcolor{cyan}15 & \cellcolor{yellow}120 & \cellcolor{light-gray}470 & \cellcolor{light-gray}1940 & \cellcolor{light-gray}3928 & \cellcolor{light-gray}11768 & 16824 & 42474 & 50598 & 114936 & 123060\\
15 & \cellcolor{cyan}16 & \cellcolor{yellow}106 & \cellcolor{light-gray}576 & \cellcolor{light-gray}1576 & \cellcolor{light-gray}5504 & \cellcolor{light-gray}9416 & 26240 & 35066 & 85664 & 99404 & 222464\\
16 & \cellcolor{gray}15 & \cellcolor{yellow}121 & \cellcolor{light-gray}485 & \cellcolor{light-gray}2061 & \cellcolor{light-gray}4413 & \cellcolor{light-gray}13829 & 21237 & 56303 & 71835 & 171239 & 194895\\
17 & \cellcolor{gray}16 & \cellcolor{yellow}107 & \cellcolor{light-gray}592 & \cellcolor{light-gray}1683 & \cellcolor{light-gray}6096 & \cellcolor{light-gray}11099 & 32336 & 46165 & 118000 & 145569 & 340464\\
18 & \cellcolor{gray}15 & \cellcolor{yellow}122 & \cellcolor{light-gray}500 & \cellcolor{light-gray}2183 & \cellcolor{light-gray}4913 & \cellcolor{light-gray}16012 & 26150 & 72315 & 97985 & 243554 & 292880\\
19 & \cellcolor{gray}16 & \cellcolor{yellow}108 & \cellcolor{light-gray}608 & \cellcolor{light-gray}1791 & \cellcolor{light-gray}6704 & \cellcolor{light-gray}12890 & 39040 & 59055 & 157040 & 204624 & 497504\\
20 & \cellcolor{gray}15 & \cellcolor{yellow}123 & \cellcolor{light-gray}515 & \cellcolor{light-gray}2306 & \cellcolor{light-gray}5428 & \cellcolor{light-gray}18318 & 31578 & 90633 & 129563 & 334187 & 422443\\
21 & \cellcolor{gray}16 & \cellcolor{yellow}109 & \cellcolor{light-gray}624 & \cellcolor{light-gray}1900 & \cellcolor{light-gray}7328 & \cellcolor{light-gray}14790 & 46368 & 73845 & 203408 & 278469 & 700912\\
22 & \cellcolor{gray}15 & \cellcolor{yellow}124 & \cellcolor{light-gray}530 & \cellcolor{light-gray}2430 & \cellcolor{light-gray}5958 & \cellcolor{light-gray}20748 & 37536 & 111381 & 167099 & 445568 & 589542\\
23 & \cellcolor{gray}16 & \cellcolor{yellow}110 & \cellcolor{light-gray}640 & \cellcolor{light-gray}2010 & \cellcolor{light-gray}7968 & 16800 & 54336 & 90645 & 257744 & 369114 & 958656\\
24 & \cellcolor{gray}15 & \cellcolor{yellow}125 & \cellcolor{light-gray}545 & \cellcolor{light-gray}2555 & \cellcolor{light-gray}6503 & 23303 & 44039 & 134684 & 211138 & 580252 & 800680\\
25 & \cellcolor{gray}16 & \cellcolor{yellow}111 & \cellcolor{light-gray}656 & \cellcolor{light-gray}2121 & \cellcolor{light-gray}8624 & 18921 & 62960 & 109566 & 320704 & 478680 & 1279360\\
26 & \cellcolor{gray}15 & \cellcolor{yellow}126 & \cellcolor{light-gray}560 & \cellcolor{light-gray}2681 & \cellcolor{light-gray}7063 & 25984 & 51102 & 160668 & 262240 & 740920 & 1062920\\
27 & \cellcolor{gray}16 & \cellcolor{yellow}112 & \cellcolor{light-gray}672 & \cellcolor{light-gray}2233 & \cellcolor{light-gray}9296 & 21154 & 72256 & 130720 & 392960 & 609400 & 1672320\\
28 & \cellcolor{gray}15 & \cellcolor{yellow}127 & \cellcolor{light-gray}575 & \cellcolor{light-gray}2808 & \cellcolor{light-gray}7638 & 28792 & 58740 & 189460 & 320980 & 930380 & 1383900\\
29 & \cellcolor{gray}16 & \cellcolor{yellow}113 & \cellcolor{light-gray}688 & \cellcolor{light-gray}2346 & 9984 & 23500 & 82240 & 154220 & 475200 & 763620 & 2147520\\
30 & \cellcolor{gray}15 & \cellcolor{yellow}128 & \cellcolor{light-gray}590 & \cellcolor{light-gray}2936 & 8228 & 31728 & 66968 & 221188 & 387948 & 1151568 & 1771848\\
31 & \cellcolor{gray}16 & \cellcolor{yellow}114 & \cellcolor{light-gray}704 & \cellcolor{light-gray}2460 & 10688 & 25960 & 92928 & 180180 & 568128 & 943800 & 2715648\\
32 & \cellcolor{gray}15 & \cellcolor{yellow}129 & \cellcolor{light-gray}605 & \cellcolor{light-gray}3065 & 8833 & 34793 & 75801 & 255981 & 463749 & 1407549 & 2235597\\
33 & \cellcolor{gray}16 & \cellcolor{yellow}115 & \cellcolor{light-gray}720 & \cellcolor{light-gray}2575 & 11408 & 28535 & 104336 & 208715 & 672464 & 1152515 & 3388112\\
34 & \cellcolor{gray}15 & \cellcolor{yellow}130 & \cellcolor{light-gray}620 & \cellcolor{light-gray}3195 & 9453 & 37988 & 85254 & 293969 & 549003 & 1701518 & 2784600\\
35 & \cellcolor{gray}16 & \cellcolor{yellow}116 & \cellcolor{light-gray}736 & \cellcolor{light-gray}2691 & 12144 & 31226 & 116480 & 239941 & 788944 & 1392456 & 4177056\\
36 & \cellcolor{gray}15 & \cellcolor{yellow}131 & \cellcolor{light-gray}635 & \cellcolor{light-gray}3326 & 10088 & 41314 & 95342 & 335283 & 644345 & 2036801 & 3428945\\
37 & \cellcolor{gray}16 & \cellcolor{yellow}117 & \cellcolor{light-gray}752 & \cellcolor{light-gray}2808 & 12896 & 34034 & 129376 & 273975 & 918320 & 1666431 & 5095376\\
38 & \cellcolor{gray}15 & \cellcolor{yellow}132 & \cellcolor{light-gray}650 & \cellcolor{light-gray}3458 & 10738 & 44772 & 106080 & 380055 & 750425 & 2416856 & 4179370\\
39 & \cellcolor{gray}16 & \cellcolor{yellow}118 & \cellcolor{light-gray}768 & \cellcolor{light-gray}2926 & 13664 & 36960 & 143040 & 310935 & 1061360 & 1977366 & 6156736\\
40 & \cellcolor{gray}15 & \cellcolor{yellow}133 & \cellcolor{light-gray}665 & \cellcolor{light-gray}3591 & 11403 & 48363 & 117483 & 428418 & 867908 & 2845274 & 5047278\\
41 & \cellcolor{gray}16 & \cellcolor{yellow}119 & \cellcolor{light-gray}784 & \cellcolor{light-gray}3045 & 14448 & 40005 & 157488 & 350940 & 1218848 & 2328306 & 7375584\\
42 & \cellcolor{gray}15 & \cellcolor{yellow}134 & \cellcolor{light-gray}680 & \cellcolor{light-gray}3725 & 12083 & 52088 & 129566 & 480506 & 997474 & 3325780 & 6044752\\
43 & \cellcolor{gray}16 & \cellcolor{yellow}120 & \cellcolor{light-gray}800 & \cellcolor{light-gray}3165 & 15248 & 43170 & 172736 & 394110 & 1391584 & 2722416 & 8767168\\
44 & \cellcolor{gray}15 & \cellcolor{yellow}135 & \cellcolor{light-gray}695 & \cellcolor{light-gray}3860 & 12778 & 55948 & 142344 & 536454 & 1139818 & 3862234 & 7184570\\
45 & \cellcolor{gray}16 & \cellcolor{yellow}121 & \cellcolor{light-gray}816 & \cellcolor{light-gray}3286 & 16064 & 46456 & 188800 & 440566 & 1580384 & 3162982 & 10347552\\
46 & \cellcolor{gray}15 & \cellcolor{yellow}136 & \cellcolor{light-gray}710 & \cellcolor{light-gray}3996 & 13488 & 59944 & 155832 & 596398 & 1295650 & 4458632 & 8480220\\
47 & \cellcolor{gray}16 & \cellcolor{yellow}122 & \cellcolor{light-gray}832 & \cellcolor{light-gray}3408 & 16896 & 49864 & 205696 & 490430 & 1786080 & 3653412 & 12133632\\
48 & \cellcolor{gray}15 & \cellcolor{yellow}137 & \cellcolor{light-gray}725 & \cellcolor{light-gray}4133 & 14213 & 64077 & 170045 & 660475 & 1465695 & 5119107 & 9945915\\
49 & \cellcolor{gray}16 & \cellcolor{yellow}123 & \cellcolor{light-gray}848 & \cellcolor{light-gray}3531 & 17744 & 53395 & 223440 & 543825 & 2009520 & 4197237 & 14143152\\
50 & \cellcolor{gray}15 & \cellcolor{yellow}138 & \cellcolor{light-gray}740 & \cellcolor{light-gray}4271 & 14953 & 68348 & 184998 & 728823 & 1650693 & 5847930 & 11596608\\
\hline \end{tabular}}

\tablelegend
\end{center}

 \newpage
\definecolor{light-gray}{gray}{0.80}

{\begin{center} \(k_{4}(\Gr_d(\R^{d+c}))\)\\ \medskip {\tiny
\begin{tabular}{|c|cccccccccc|}\hline
\diagbox{$d$}{$c$} & 1 & 2 & 3 & 4 & 5 & 6 & 7 & 8 & 9 & 10\\ \hline
1 & \cellcolor{cyan}2 & \cellcolor{cyan}3 & \cellcolor{cyan}4 & \cellcolor{cyan}5 & \cellcolor{cyan}6 & \cellcolor{cyan}7 & \cellcolor{cyan}8 & \cellcolor{cyan}9 & \cellcolor{cyan}10 & \cellcolor{cyan}11\\
2 & \cellcolor{cyan}3 & \cellcolor{cyan}6 & \cellcolor{cyan}10 & \cellcolor{cyan}15 & \cellcolor{cyan}21 & \cellcolor{cyan}28 & \cellcolor{cyan}36 & \cellcolor{cyan}45 & \cellcolor{cyan}55 & \cellcolor{cyan}66\\
3 & \cellcolor{cyan}4 & \cellcolor{cyan}10 & \cellcolor{cyan}20 & \cellcolor{cyan}35 & \cellcolor{cyan}56 & \cellcolor{cyan}84 & \cellcolor{cyan}120 & \cellcolor{cyan}165 & \cellcolor{cyan}220 & \cellcolor{cyan}286\\
4 & \cellcolor{cyan}5 & \cellcolor{cyan}15 & \cellcolor{cyan}35 & \cellcolor{cyan}70 & \cellcolor{cyan}126 & \cellcolor{cyan}210 & \cellcolor{cyan}330 & \cellcolor{green}495 & \cellcolor{green}715 & \cellcolor{green}1001\\
5 & \cellcolor{cyan}6 & \cellcolor{cyan}21 & \cellcolor{cyan}56 & \cellcolor{cyan}126 & \cellcolor{cyan}252 & \cellcolor{cyan}462 & \cellcolor{green}792 & \cellcolor{green}1287 & \cellcolor{green}2002 & \cellcolor{green}3003\\
6 & \cellcolor{cyan}7 & \cellcolor{cyan}28 & \cellcolor{cyan}84 & \cellcolor{cyan}210 & \cellcolor{cyan}462 & \cellcolor{green}924 & \cellcolor{green}1716 & \cellcolor{green}3003 & \cellcolor{green}5005 & \cellcolor{green}8008\\
7 & \cellcolor{cyan}8 & \cellcolor{cyan}36 & \cellcolor{cyan}120 & \cellcolor{cyan}330 & \cellcolor{green}792 & \cellcolor{green}1716 & \cellcolor{green}3432 & \cellcolor{green}6435 & \cellcolor{green}11440 & \cellcolor{green}19448\\
8 & \cellcolor{cyan}9 & \cellcolor{cyan}45 & \cellcolor{cyan}165 & \cellcolor{green}495 & \cellcolor{green}1287 & \cellcolor{green}3003 & \cellcolor{green}6435 & \cellcolor{green}12870 & \cellcolor{green}24310 & \cellcolor{green}43758\\
9 & \cellcolor{cyan}10 & \cellcolor{cyan}55 & \cellcolor{cyan}220 & \cellcolor{green}715 & \cellcolor{green}2002 & \cellcolor{green}5005 & \cellcolor{green}11440 & \cellcolor{green}24310 & \cellcolor{green}48620 & \cellcolor{green}92378\\
10 & \cellcolor{cyan}11 & \cellcolor{cyan}66 & \cellcolor{cyan}286 & \cellcolor{green}1001 & \cellcolor{green}3003 & \cellcolor{green}8008 & \cellcolor{green}19448 & \cellcolor{green}43758 & \cellcolor{green}92378 & \cellcolor{green}184756\\
11 & \cellcolor{cyan}12 & \cellcolor{cyan}78 & \cellcolor{green}364 & \cellcolor{green}1365 & \cellcolor{green}4368 & \cellcolor{green}12376 & \cellcolor{green}31824 & \cellcolor{green}75582 & \cellcolor{green}167960 & \cellcolor{green}352716\\
12 & \cellcolor{cyan}13 & \cellcolor{cyan}91 & \cellcolor{green}455 & \cellcolor{green}1820 & \cellcolor{green}6188 & \cellcolor{green}18564 & \cellcolor{green}50388 & \cellcolor{green}125970 & \cellcolor{green}293930 & \cellcolor{green}646646\\
13 & \cellcolor{cyan}14 & \cellcolor{cyan}105 & \cellcolor{green}560 & \cellcolor{green}2380 & \cellcolor{green}8568 & \cellcolor{green}27132 & \cellcolor{green}77520 & \cellcolor{green}203490 & \cellcolor{green}497420 & \cellcolor{green}1144066\\
14 & \cellcolor{cyan}15 & \cellcolor{cyan}120 & \cellcolor{green}680 & \cellcolor{green}3060 & \cellcolor{green}11628 & \cellcolor{green}38760 & \cellcolor{green}116280 & \cellcolor{green}319770 & \cellcolor{green}817190 & \cellcolor{green}1961256\\
15 & \cellcolor{cyan}16 & \cellcolor{cyan}136 & \cellcolor{green}816 & \cellcolor{green}3876 & \cellcolor{green}15504 & \cellcolor{green}54264 & \cellcolor{green}170544 & \cellcolor{green}490314 & \cellcolor{green}1307504 & \cellcolor{green}3268760\\
16 & \cellcolor{cyan}17 & \cellcolor{green}153 & \cellcolor{green}969 & \cellcolor{green}4845 & \cellcolor{green}20349 & \cellcolor{green}74613 & \cellcolor{green}245157 & \cellcolor{green}735471 & \cellcolor{green}2042975 & \cellcolor{green}5311735\\
17 & \cellcolor{cyan}18 & \cellcolor{green}171 & \cellcolor{green}1140 & \cellcolor{green}5985 & \cellcolor{green}26334 & \cellcolor{green}100947 & \cellcolor{green}346104 & \cellcolor{green}1081575 & \cellcolor{green}3124550 & \cellcolor{green}8436285\\
18 & \cellcolor{cyan}19 & \cellcolor{green}190 & \cellcolor{green}1330 & \cellcolor{green}7315 & \cellcolor{green}33649 & \cellcolor{green}134596 & \cellcolor{green}480700 & \cellcolor{green}1562275 & \cellcolor{green}4686825 & \cellcolor{green}13123110\\
19 & \cellcolor{cyan}20 & \cellcolor{green}210 & \cellcolor{green}1540 & \cellcolor{green}8855 & \cellcolor{green}42504 & \cellcolor{green}177100 & \cellcolor{green}657800 & \cellcolor{green}2220075 & \cellcolor{green}6906900 & \cellcolor{green}20030010\\
20 & \cellcolor{cyan}21 & \cellcolor{green}231 & \cellcolor{green}1771 & \cellcolor{green}10626 & \cellcolor{green}53130 & \cellcolor{green}230230 & \cellcolor{green}888030 & \cellcolor{green}3108105 & \cellcolor{green}10015005 & \cellcolor{green}30045015\\
21 & \cellcolor{cyan}22 & \cellcolor{green}253 & \cellcolor{green}2024 & \cellcolor{green}12650 & \cellcolor{green}65780 & \cellcolor{green}296010 & \cellcolor{green}1184040 & \cellcolor{green}4292145 & \cellcolor{green}14307150 & \cellcolor{green}44352165\\
22 & \cellcolor{cyan}23 & \cellcolor{green}276 & \cellcolor{green}2300 & \cellcolor{green}14950 & \cellcolor{green}80730 & \cellcolor{green}376740 & \cellcolor{green}1560780 & \cellcolor{green}5852925 & \cellcolor{green}20160075 & 64512240\\
23 & \cellcolor{cyan}24 & \cellcolor{green}300 & \cellcolor{green}2600 & \cellcolor{green}17550 & \cellcolor{green}98280 & \cellcolor{green}475020 & \cellcolor{green}2035800 & \cellcolor{green}7888725 & 28048800 & 52240890\\
24 & \cellcolor{cyan}25 & \cellcolor{green}325 & \cellcolor{green}2925 & \cellcolor{green}20475 & \cellcolor{green}118755 & \cellcolor{green}593775 & \cellcolor{green}2629575 & 10518300 & 22789650 & 75030540\\
25 & \cellcolor{cyan}26 & \cellcolor{green}351 & \cellcolor{green}3276 & \cellcolor{green}23751 & \cellcolor{green}142506 & \cellcolor{green}736281 & 3365856 & 8625006 & 31414656 & 60865896\\
26 & \cellcolor{cyan}27 & \cellcolor{green}378 & \cellcolor{green}3654 & \cellcolor{green}27405 & \cellcolor{green}169911 & 906192 & 2799486 & 11424492 & 25589136 & 86455032\\
27 & \cellcolor{cyan}28 & \cellcolor{green}406 & \cellcolor{green}4060 & \cellcolor{green}31465 & \cellcolor{light-gray}201376 & 767746 & 3567232 & 9392752 & 34981888 & 70258648\\
28 & \cellcolor{cyan}29 & \cellcolor{green}435 & \cellcolor{green}4495 & \cellcolor{light-gray}35960 & \cellcolor{light-gray}174406 & 942152 & 2973892 & 12366644 & 28563028 & 98821676\\
29 & \cellcolor{cyan}30 & \cellcolor{green}465 & \cellcolor{light-gray}4960 & \cellcolor{light-gray}31930 & \cellcolor{light-gray}206336 & 799676 & 3773568 & 10192428 & 38755456 & 80451076\\
30 & \cellcolor{cyan}31 & \cellcolor{yellow}496 & \cellcolor{light-gray}4526 & \cellcolor{light-gray}36456 & \cellcolor{light-gray}178932 & 978608 & 3152824 & 13345252 & 31715852 & 112166928\\
31 & \cellcolor{cyan}32 & \cellcolor{yellow}466 & \cellcolor{light-gray}4992 & \cellcolor{light-gray}32396 & \cellcolor{light-gray}211328 & 832072 & 3984896 & 11024500 & 42740352 & 91475576\\
32 & \cellcolor{gray}31 & \cellcolor{yellow}497 & \cellcolor{light-gray}4557 & \cellcolor{light-gray}36953 & 183489 & 1015561 & 3336313 & 14360813 & 35052165 & 126527741\\
33 & \cellcolor{gray}32 & \cellcolor{yellow}467 & \cellcolor{light-gray}5024 & \cellcolor{light-gray}32863 & 216352 & 864935 & 4201248 & 11889435 & 46941600 & 103365011\\
34 & \cellcolor{gray}31 & \cellcolor{yellow}498 & \cellcolor{light-gray}4588 & \cellcolor{light-gray}37451 & 188077 & 1053012 & 3524390 & 15413825 & 38576555 & 141941566\\
35 & \cellcolor{gray}32 & \cellcolor{yellow}468 & \cellcolor{light-gray}5056 & \cellcolor{light-gray}33331 & 221408 & 898266 & 4422656 & 12787701 & 51364256 & 116152712\\
36 & \cellcolor{gray}31 & \cellcolor{yellow}499 & \cellcolor{light-gray}4619 & \cellcolor{light-gray}37950 & 192696 & 1090962 & 3717086 & 16504787 & 42293641 & 158446353\\
37 & \cellcolor{gray}32 & \cellcolor{yellow}469 & \cellcolor{light-gray}5088 & \cellcolor{light-gray}33800 & 226496 & 932066 & 4649152 & 13719767 & 56013408 & 129872479\\
38 & \cellcolor{gray}31 & \cellcolor{yellow}500 & \cellcolor{light-gray}4650 & \cellcolor{light-gray}38450 & 197346 & 1129412 & 3914432 & 17634199 & 46208073 & 176080552\\
39 & \cellcolor{gray}32 & \cellcolor{yellow}470 & \cellcolor{light-gray}5120 & \cellcolor{light-gray}34270 & 231616 & 966336 & 4880768 & 14686103 & 60894176 & 144558582\\
40 & \cellcolor{gray}31 & \cellcolor{yellow}501 & \cellcolor{light-gray}4681 & \cellcolor{light-gray}38951 & 202027 & 1168363 & 4116459 & 18802562 & 50324532 & 194883114\\
41 & \cellcolor{gray}32 & \cellcolor{yellow}471 & \cellcolor{light-gray}5152 & \cellcolor{light-gray}34741 & 236768 & 1001077 & 5117536 & 15687180 & 66011712 & 160245762\\
42 & \cellcolor{gray}31 & \cellcolor{yellow}502 & \cellcolor{light-gray}4712 & \cellcolor{light-gray}39453 & 206739 & 1207816 & 4323198 & 20010378 & 54647730 & 214893492\\
43 & \cellcolor{gray}32 & \cellcolor{yellow}472 & \cellcolor{light-gray}5184 & \cellcolor{light-gray}35213 & 241952 & 1036290 & 5359488 & 16723470 & 71371200 & 176969232\\
44 & \cellcolor{gray}31 & \cellcolor{yellow}503 & \cellcolor{light-gray}4743 & \cellcolor{light-gray}39956 & 211482 & 1247772 & 4534680 & 21258150 & 59182410 & 236151642\\
45 & \cellcolor{gray}32 & \cellcolor{yellow}473 & \cellcolor{light-gray}5216 & \cellcolor{light-gray}35686 & 247168 & 1071976 & 5606656 & 17795446 & 76977856 & 194764678\\
46 & \cellcolor{gray}31 & \cellcolor{yellow}504 & \cellcolor{light-gray}4774 & \cellcolor{light-gray}40460 & 216256 & 1288232 & 4750936 & 22546382 & 63933346 & 258698024\\
47 & \cellcolor{gray}32 & \cellcolor{yellow}474 & \cellcolor{light-gray}5248 & \cellcolor{light-gray}36160 & 252416 & 1108136 & 5859072 & 18903582 & 82836928 & 213668260\\
48 & \cellcolor{gray}31 & \cellcolor{yellow}505 & \cellcolor{light-gray}4805 & \cellcolor{light-gray}40965 & 221061 & 1329197 & 4971997 & 23875579 & 68905343 & 282573603\\
49 & \cellcolor{gray}32 & \cellcolor{yellow}475 & \cellcolor{light-gray}5280 & \cellcolor{light-gray}36635 & 257696 & 1144771 & 6116768 & 20048353 & 88953696 & 233716613\\
50 & \cellcolor{gray}31 & \cellcolor{yellow}506 & \cellcolor{light-gray}4836 & \cellcolor{light-gray}41471 & 225897 & 1370668 & 5197894 & 25246247 & 74103237 & 307819850\\
\hline \end{tabular}}

\tablelegend
\end{center}

\end{document}